\documentclass[a4paper]{article}
\usepackage{amssymb,amsmath,amsthm,enumerate,graphicx,slashed,hyperref,tikz}
\usetikzlibrary{arrows.meta}

\newcommand{\RR}{{\mathbb{R}}}
\newcommand{\CC}{{\mathbb{C}}}
\newcommand{\ZZ}{{\mathbb{Z}}}

\newcommand{\NN}{{\mathbb{N}}}
\newcommand{\CP}{{\mathbb{CP}}}
\newcommand{\pa}{\partial}
\newcommand{\ii}{{\rm i}}
\newcommand{\dd}{{\rm d}}

\newcommand{\sfrac}[2]{{\textstyle\frac{#1}{#2}}}
\newcommand{\Tr}{\mathrm{Tr}}
\newcommand{\End}{\operatorname{End}}
\newcommand{\pardeg}{\operatorname{pardeg}}

\newcommand{\diag}{\operatorname{diag}}
\newcommand{\Sh}{\Sigma}
\newcommand{\RS}{M}
\newcommand{\U}{\mathrm{U}}
\renewcommand{\mod}{\;\operatorname{mod}\;}

\newcommand{\doiref}[2]{\href{http://dx.doi.org/#1}{#2}}
\newcommand{\arxivref}[1]{\href{http://arxiv.org/abs/#1}{arXiv:#1}}

\theoremstyle{plain}
\newtheorem{theorem}{Theorem}
\newtheorem{lemma}[theorem]{Lemma}
\newtheorem{proposition}[theorem]{Proposition}
\newtheorem{corollary}[theorem]{Corollary}
\theoremstyle{definition}
\newtheorem{definition}[theorem]{Definition}

\title{Parabolic Higgs bundles and\\ cyclic monopole chains}
\author{Derek Harland\footnote{email address: d.g.harland@leeds.ac.uk}
  \bigskip
  \\School of Mathematics,
  \\University of Leeds,
  \\LS2 9JT}
\date{23rd May 2021}

\begin{document}

\maketitle

\abstract{We formulate a correspondence between SU(2) monopole chains and ``spectral data'', consisting of curves in $\CP^1\times\CP^1$ equipped with parabolic line bundles.  This is the analogue for monopole chains of Donaldson's association of monopoles with rational maps.  The construction is based on the Nahm transform, which relates monopole chains to Higgs bundles on the cylinder.  As an application, we classify  charge $k$ monopole chains which are invariant under actions of $\ZZ_{2k}$.  We present images of these symmetric monopole chains that were constructed using a numerical Nahm transform.}

\section{Introduction}

This paper concerns monopole chains.  An SU(2) monopole chain is a pair $(A,\Phi)$ consisting of an SU(2) connection $A$ and an $\mathfrak{su}(2)$-valued section $\Phi$ over $\RR^2\times S^1$ satisfying the Bogomolny equation,
\begin{equation}
\label{bog eq}
F^{A}=\ast\dd^{A}\Phi,
\end{equation}
and the boundary conditions, 
\begin{equation}
|F^A|\to 0,\quad|\Phi|\sim v+u\ln\rho\text{ as }\rho\to\infty.
\end{equation}
Here $u,v\in\RR$, $u>0$ and $\rho$ is the radial coordinate on $\RR^2$.  The second boundary condition implies that $\Phi$ is non-vanishing for large $\rho$, so defines a map from a torus to $\RR^3\setminus\{0\}$ with well-defined degree $k\in\ZZ$.  In fact $k$ is non-negative and is related to $u$ via $u=k/\beta$, with $\beta$ being the circumference of $S^1$.  For a more detailed discussion of the boundary conditions see \cite{ck,mochizuki}.  There is a fairly substantial literature on monopole chains \cite{ck,ckagain,foscolo1,foscolo2,hw1,hw2,maldonado1,maldonado2,maldonado3,maldonadoward,mochizuki}. Like monopoles on $\RR^3$, monopole chains form moduli spaces which are known to be hyperk\"ahler \cite{foscolo2}.  Monopole chains can be identified with difference modules via a Kobayashi-Hitchin correspondence \cite{mochizuki}.

In the first part of this article we describe results which allow one to study the moduli spaces of monopole chains fairly directly.  The first of these, theorem \ref{thm:KH}, is a Kobayashi--Hitchin correspondence between certain moduli spaces $\mathcal{M}_{\text{Higgs}}^k$ of parabolic Higgs bundles on $\CP^1$ and moduli spaces $\mathcal{M}_{\text{Hit}}^k$ of solutions of Hitchin's equations on the cylinder.  The latter correspond to moduli spaces $\mathcal{M}_{\text{mon}}^k$ of monopole chains, via the Nahm transform of Cherkis--Kapustin \cite{ck}.  The second, theorem \ref{thm:spectral}, gives a bijection between the moduli spaces $\mathcal{M}_{\text{Higgs}}^k$ and moduli spaces $\mathcal{M}_{\text{spec}}^k$ of spectral curves equipped with parabolic line bundles.

The composition $\mathcal{M}_{\text{spec}}^k\to\mathcal{M}^k_{\text{Higgs}}\to\mathcal{M}^k_{\text{Hit}}\to\mathcal{M}_{\text{mon}}^k$ can be viewed as the analogue for monopole chains of Donaldson's \cite{donaldson} and Jarvis' \cite{jarvis} maps from moduli spaces of rational maps to moduli spaces of monopoles on $\RR^3$.  Spectral data consist of algebraic curves and line bundles and, like rational maps, are much easier to write down than explicit solutions of \eqref{bog eq}.  This is what makes theorems \ref{thm:KH} and \ref{thm:spectral} useful.  Donaldson's rational maps were used to great effect by Segal and Selby in their work on Sen's conjectures \cite{segalselby}, and similarly our result could perhaps be used to study the cohomology of moduli spaces of monopole chains.

In the second part of this article we use these theorems to study monopole chains with a high degree of symmetry.  We consider cyclic groups $\ZZ_{m}^{(n)}$ generated by the following maps $\RR^2\times S^1\to\RR^2\times S^1$:
\begin{equation}
\RR^2\times S^1 = \CC\times(\RR/\ZZ) \ni (\zeta,\chi+\ZZ)\mapsto \left(e^{2\pi\ii/m}\zeta,\chi+\frac{n}{m}+\ZZ\right).
\end{equation}
By classifying invariant spectral data, we show in corollary \ref{cor:classification} that for each $k>0$ and $0\leq l<k$ there exists a monopole chain of charge $k$ invariant under the action of $\ZZ_{2k}^{(2l)}$.  Assuming that the map $\mathcal{M}_{\text{Higgs}}^k\to\mathcal{M}_{\text{mon}}^k$ is surjective, these are the only $\ZZ_{2k}^{(n)}$-invariant monopole chains in $\mathcal{M}_{\text{mon}}^k$.  We have also been able to construct these monopole chains numerically, and images are presented near the end of this article.

Our work on symmetric monopole chains can be viewed as the analogue for monopole chains of constructions of symmetric monopoles obtained in the 1990s \cite{hmm,hs1,hs2}.  It is also motivated by several papers that constructed cyclic-symmetric monopole chains using ad hoc methods \cite{maldonado1,maldonado2,maldonadoward}.  Our classification includes several new examples of symmetric monopole chains which were not accessible using the methods of these papers.  Our work here parallels recent work of Cork \cite{cork}, who classified cyclic-symmetric calorons (instantons on $\RR^3\times S^1$).  Additional motivation comes from the paper \cite{hw1}, which constructed minimisers of the Skyrme energy on $\RR^2\times S^1$.  These energy minimisers turned out to be invariant under groups of the form $\ZZ_m^{(n)}$.  A well-known heuristic says that skyrmions resemble monopoles, so studying $\ZZ_m^{(n)}$-invariant monopole chains presents a more systematic way to study minimisers of the Skyrme energy on $\RR^2\times S^1$.

We now outline the contents of this article.  In section \ref{sec:2} we describe in detail the Higgs bundles on $\CP^1$ that correspond to monopole chains.  They are particular examples of filtered Higgs bundles.  Theorem \ref{thm:KH}, which relates these to Hitchin's equations on the cylinder, is a particular case of a Kobayashi--Hitchin correspondence estabished in \cite{mochizuki2011}.  To keep this article relatively self-contained we given an independent proof of parts of this theorem in an appendix.

In section \ref{sec:3} we prove theorem \ref{thm:spectral}, which relates filtered (or parabolic) Higgs bundles on $\CP^1$ to spectral curves equipped with parabolic line bundles.  Although most of this material is fairly standard in the Higgs bundle literature, the inclusion of parabolic structures may not be.

In section \ref{sec:4} we prove our main result, theorem \ref{thm:cyclic spectral}, which leads to a classification (in corollary \ref{cor:classification}) of monopole chains invariant under actions of $\ZZ_{2k}^{(n)}$.  This section also includes a discussion of various groups which act naturally on the moduli spaces.  The proof of theorem \ref{thm:cyclic spectral} uses standard tools, such as the Abel--Jacobi map, but is technically rather intricate.

In section \ref{sec:5} we write down the Higgs bundles that correspond to these symmetric monopole chains.  Although the existence of these Higgs bundles is guaranteed by theorems \ref{thm:spectral} and \ref{thm:cyclic spectral}, writing them down explicitly does not appear to be straightforward.  From these explicit Higgs bundles we have been able to construct the associated monopole chains through numerical implementations of the Kobayashi--Hitchin correspondence and Nahm transform.  Pictures of these monopole chains are presented at the end of section \ref{sec:5}.

\section{Parabolic structures and Higgs bundles}
\label{sec:2}

In this section we describe parabolic Higgs bundles associated with monopole chains.  We begin by defining and describing basic properties of parabolic bundles.  Although this material is standard, our presentation is not, and readers are encouraged to review this even if they are familiar with parabolic bundles.  We then describe in some detail the parabolic Higgs bundles relevant to monopole chains, and finally prove a Kobayashi--Hitchin correspondence relating these to monopole chains.

\subsection{Parabolic vector bundles}
\begin{definition}
Let $E\to \RS$ be a rank $k$ holomorphic vector bundle over a Riemann surface $\RS$ and let $P\in\RS$.  A \emph{parabolic structure at $P$} is a filtration of the fibre at $P$,
\begin{equation}
0 = E_P^{k_P}\subset E_P^{k_P-1} \subset \ldots\subset E_P^1 \subset E_P^0 = E_P,
\end{equation}
together with real numbers $\alpha_P^{k_P-1},\ldots,\alpha_P^0$ (called \emph{weights}) satisfying
\begin{equation}
\label{alpha order}
\alpha_P^0+1>\alpha_P^{k_P-1}>\alpha_P^{k_P-2}>\ldots>\alpha_P^0.
\end{equation}
A parabolic structure is called \emph{full} if $k_P=k$, and in this case the quotient spaces $E_P^i/E_P^{i+1}$ are all one-dimensional.  A \emph{framed parabolic vector bundle} consists of a holomorphic vector bundle $E\to\RS$ together with parabolic structures at a finite set $\mathcal{P}$ of points.

A local holomorphic frame $e_0,\ldots,e_{k-1}$ near a parabolic point $P$ is said to be \emph{compatible} if, for each $0\leq i<k_P$, there exists integers $m_P^i$ such that $E_P^i=\operatorname{span}\{e_j(P)\::\:m_P^i\leq j<k\}$.  In the particular case of a full parabolic structure, this means that $E_P^i = \operatorname{span}\{e_{i}(P),e_{i+1}(P),\ldots,e_{k-1}(P)\}$.  A holomorphic trivialisation is called compatible if it is induced by a compatible frame.
\end{definition}

The sheaf of holomorphic sections of a framed parabolic vector bundle $E$ will be denoted $\mathcal{E}$, and for any parabolic point $P$ the sheaf of holomorphic sections $\sigma$ such that $\sigma(P)\in E^i_P$ will be denoted $\mathcal{E}^{iP}$.

\begin{definition} Let $E$ be a framed parabolic vector bundle.  A hermitian metric $h$ on $E\setminus \mathcal{P}$ is said to be \emph{compatible with the holomorphic structure} if near each parabolic point $P$ there exists a holomorphic coordinate $z$ on a neighbourhood $U$ of $P$ with $z(P)=0$ such that for all $0<i\leq k$ and all $\alpha^{i-1}_P<\alpha\leq \alpha^i_P$,
\begin{equation}
\label{metric and weights}
\big\{ s\in\mathcal{E}|_U\::\: h(s,s) = O(|z|^{2\alpha})\big\} = \mathcal{E}^{iP}|_U.
\end{equation}
\end{definition}

Parabolic structures are a way to encode monodromy of a connection.  To see this, consider the singular hermitian metric on the trivial rank $k$ vector bundle over $\CC$ defined by the matrix
\begin{equation*}
h = \diag(|z|^{2\alpha_0^0},|z|^{2\alpha_0^1},\ldots,|z|^{2\alpha_0^{k-1}}).
\end{equation*}
This is compatible with the parabolic structure at zero in which $E_0^i$ consists of vectors whose first $i$ entries vanish.  The Chern connection of this metric is
\begin{equation*}
A = h^{-1}\pa h = \frac{\dd z}{z}\diag(\alpha_0^0,\alpha_0^1,\ldots,\alpha_0^{k-1}).
\end{equation*}
The gauge transformation $g=\diag(|z|^{-\alpha_0^0},\ldots,|z|^{-\alpha_0^{k-1}})$ brings us to a unitary gauge, in which
\begin{equation*}
A \mapsto g^{-1}\dd g + g^{-1}Ag = \ii\dd\theta \diag(\alpha_0^0,\alpha_0^1,\ldots,\alpha_0^{k-1}),
\end{equation*}
with $\theta=\arg z$.  In this form the connection clearly has holonomy around $z=0$ described by $\alpha_0^i$.

Given a framed parabolic bundle $E$ with parabolic structure at $P$, we define sheaves $\mathcal{E}^{nP}$ for all $n\in\ZZ$ as follows.  Let $U\subset\RS$ be an open set.  If $P\notin U$ set $\mathcal{E}^{nP}|_U=\mathcal{E}_U$.  If $P\in U$, let $z_P$ be a local coordinate that vanishes at $P$, and let $q,r$ be integers with $0\leq r<k_P$ such that $n=qk_P+r$.  Let $\mathcal{E}^{nP}|_U$ be the set of meromorphic sections $\sigma$ of $E$ over $U$ such that $z_P^{\;-q}\sigma$ is holomorphic near $P$ and $(z_P^{\;-q}\sigma)(P)\in E^r_P$.  Then there is a filtration
\begin{equation}
\cdots \mathcal{E}^{(k_P+1)P}\subset \mathcal{E}^{k_P P}\subset \mathcal{E}^{(k_P-1)P}\subset\cdots\subset \mathcal{E}^{1P}\subset\mathcal{E}\subset\mathcal{E}^{-1P}\subset\cdots.
\end{equation}
We also define associated weights $\alpha_P^n = \alpha_P^r+q$.  Then
\begin{equation}
\ldots>\alpha_P^{k_P+1}>\alpha_P^{k_P}>\alpha_P^{k_P-1}>\ldots> \alpha_P^1>\alpha_P^0>\alpha_P^{-1}>\ldots,
\end{equation}
If $h$ is any compatible metric and $U\subset\RS$ is an open set containing $P$ then $\mathcal{E}^{nP}|_U$ is the set of holomorphic sections $s$ of $E|_{U\setminus\{P\}}$ such that $h(s,s)=O(|z_P|^{2\alpha_P^n})$.


Each of the sheaves $\mathcal{E}^{nP}$ is locally free\footnote{For example, if $E\to\CC$ has a full parabolic structure at $z=0$ and $e_0,\ldots,e_{k-1}$ is a compatible frame near $z=0$, then $e_1,\ldots,e_{k-1},\frac{1}{z} e_0$ is a frame for $\mathcal{E}^{1P}$.} and of the same rank as $\mathcal{E}$.  The holomorphic vector bundle $F$ defined by the locally free sheaf $\mathcal{E}^{nP}$, such that $\mathcal{F}\cong\mathcal{E}^{nP}$, is called the \emph{twist} of $E$ by $nP$.  This bundle carries a natural parabolic structure at $P$ such that
\begin{equation}
\mathcal{F}^{mP} \cong \mathcal{E}^{(n+m)P}\quad \forall m\in\ZZ
\end{equation}
The associated weights are
\begin{equation}
\label{twist weights}
\beta_P^m = \alpha_P^{n+m}.
\end{equation}
Since by definition $\mathcal{F}|_{\RS\setminus\{P\}}=\mathcal{E}|_{\RS\setminus\{P\}}$, the bundles $F|_{\RS\setminus\{P\}}$ and $E|_{\RS\setminus\{P\}}$ are canonically isomorphic.  In the case of line bundles $L$ twisting is equivalent to tensoring with the line bundle associated with the divisor $-nP$ and adding $n$ to the parabolic weights $\alpha_P^i$; we use the notation $\mathcal{L}^{nP}=\mathcal{L}[-nP]$ in this case.

Similarly, for any divisor $\sum_{P\in\mathcal{P}}n_P P$ supported in the set $\mathcal{P}$ of parabolic points, one can define a twist $\mathcal{E}^{\sum_P n_P P}$ of $\mathcal{E}$. Twisting defines an action of the free abelian group $\ZZ^{\mathcal{P}}$ generated by the parabolic points on the moduli space of framed parabolic vector bundles, and the associated equivalence classes are called \emph{unframed parabolic vector bundles}.  An unframed parabolic vector bundle has a unique representative such that the parabolic weights all lie in the interval $[0,1)$, so unframed parabolic vector bundles can equivalently be defined as framed parabolic vector bundles satisfying this constraint.  Most literature refers to what we have called unframed parabolic vector bundles as ``filtered sheaves'' \cite{mochizuki2011,mochizuki} or simply ``parabolic vector bundles''; we have introduced the distinction between framed and unframed bundles for later convenience.

If $h$ is a hermitian metric on $E|_{\RS\setminus\{\mathcal{P}\}}$ and $F$ is a twist of $E$ then $h$ induces a metric on $F|_{\RS\setminus\{\mathcal{P}\}}$, because $F|_{\RS\setminus\{\mathcal{P}\}}\cong E|_{\RS\setminus\{\mathcal{P}\}}$ canonically.  The weights \eqref{twist weights} have been defined in such a way that $h$ is compatible with the parabolic structure on $F$ if and only if it is compatible with the parabolic structure on $E$.


An important invariant of framed parabolic vector bundles is the \emph{parabolic degree}, defined by
\begin{equation}
\pardeg(E) = \deg(E)+\sum_{P\in\mathcal{P}}\sum_{i=0}^{k_P-1}\alpha_P^i\dim(E_{i}/E_{i+1}).
\end{equation}
This is invariant under twist, so descends to an invariant of unframed parabolic vector bundles.

\subsection{Parabolic Higgs bundles and Hitchin's equations}

To construct monopole chains we will need a parabolic vector bundle $E$ and a meromorphic section $\phi$ of $\mathrm{End}(E)$.  The section $\phi$ will have prescribed behaviour at the parabolic points, and the next proposition summarises this behaviour.
\begin{proposition}\label{prop:HBC}
Let $E$ be a framed parabolic vector bundle of rank $k$ with full parabolic structure at a parabolic point $P$.  Let $\phi$ be a meromorphic section of $\mathrm{End}(E)$ which is holomorphic away from the parabolic points.  Then the following are equivalent:
\begin{enumerate}[(i)]
\item\label{HBC1} $\phi$ induces isomorphisms of stalks $(\mathcal{E}^{nP})_P\to(\mathcal{E}^{(n-1)P})_P$ for all $n\in\ZZ$.
\item\label{HBC2} $\phi$ has a simple pole at $P$.  The residue $\mathrm{Res}_P\phi$ of $\phi$ at $P$ is a surjective map $E_P^0\to E_P^{k-1}$, and $\phi$ induces isomorphisms $E_P^i\to E_P^{i-1}/E_P^{k-1}$ for $0<i<k$.
\item\label{HBC3} Let $z$ be any local holomorphic coordinate that vanishes at $P$.  There is a compatible local holomorphic frame $e_0,\ldots,e_{k-1}$ for $E$ such that the matrix of $\phi$ with respect to this frame takes the form
\begin{equation}
\label{eq:HBC3}
\phi = \left( \begin{array}{c|ccc} -f_{k-1}(z) & & &  \\ \vdots & & \mathrm{Id}_{k-1} & \\ -f_{1}(z) & & & \\ \hline c/z-f_0(z) & 0& \cdots&0 \end{array} \right)
\end{equation}
for some holomorphic functions $f_i(z)$ and non-zero $c\in\CC$.
\end{enumerate}
\end{proposition}
\begin{proof}
First we show (\ref{HBC1}) implies (\ref{HBC2}).  Choose a holomorphic coordinate $z$ that vanishes at $P$.  Since $z\phi(\mathcal{E}_P)\subseteq (\mathcal{E}^{(k-1)P})_P\subseteq \mathcal{E}_P$, $z\phi$ maps local holomorphic sections to local holomorphic sections, and so $\phi$ has a simple pole.  The filtration of the fibre of $E$ at $P$ is recovered from the stalks of the sheaves $\mathcal{E}^{nP}$ by setting $E_P^i = (\mathcal{E}^{iP})_P/(\mathcal{E}^{kP})_P$ for $0\leq i\leq k$.  Since $\phi$ induces isomorphisms $(\mathcal{E}^{nP})_P\to(\mathcal{E}^{(n-1)P})_P$ it also induces isomorphisms $E_P^i = (\mathcal{E}^{iP})_P/(\mathcal{E}^{kP})_P \to (\mathcal{E}^{(i-1)P})_P/(\mathcal{E}^{(k-1)P})_P=E_P^{i-1}/E_P^{k-1}$.  Finally, $z\phi$ induces a map $E_P^0 = (\mathcal{E}^{0P})_P/(\mathcal{E}^{kP})_P \to (\mathcal{E}^{(k-1)P})_P/(\mathcal{E}^{(2k-1)P})_P\to(\mathcal{E}^{(k-1)P})_P/(\mathcal{E}^{kP})_P=E_P^{k-1}$ which is a composition of surjections, hence the residue of $\phi$ is a surjective map $E_P^0\to E_P^{k-1}$.

Next we show (\ref{HBC2}) implies (\ref{HBC3}).  Let $e_{k-1}$ be a local non-vanishing holomorphic section of $E$ such that $e_{k-1}(P)\in E_P^{k-1}$.  Let $e_{k-1-i}=\phi^ie_{k-1}$ for $0< i<k$.  Since $\phi(E_P^i)\subset E_P^{i-1}$, $e_i$ are all holomorphic in a neighbourhood of $P$ and $e_i(P)\in E_P^i$.  Since the maps $\phi:E_P^i\to E_P^{i-1}/E_P^{k-1}$ are isomorphisms, $e_i(P),\ldots,e_{k-1}(P)$ form a basis for $E_P^i$ and moreover $e_0,\ldots,e_{k-1}$ form a local frame for $E$.  Since $\mathrm{Res}_P\phi:E_P^0\to E_P^{k-1}$ is surjective and $\mathrm{Res}_P\phi(e_i)=0$ for $0<i<k$, it must be that $\mathrm{Res}_P(e_0)=ce_{k-1}$ for some $c\neq 0$.  It follows that the matrix of $\phi$ has the stated form with respect to this frame.

Finally we show (\ref{HBC3}) implies (\ref{HBC1}).  It is clear from the matrix form of $\phi$ that $\phi((\mathcal{E}^{nP})_P)\subseteq (\mathcal{E}^{(n-1)P})_P$.  To show that the map $\phi:(\mathcal{E}^{nP})_P\to(\mathcal{E}^{(n-1)P})_P$ is an isomorphism we just need to exhibit an inverse.  The inverse of the matrix given for $\phi$ is
\begin{equation}
\label{phi normal}
\phi^{-1} = \left( \begin{array}{ccc|c} 0&\cdots&0&z(c-zf_0(z))^{-1} \\ \hline & & & zf_{k-1}(z)(c-zf_0(z))^{-1} \\ &\mathrm{Id}_{k-1}&&\vdots \\ &&&zf_1(z)(c-zf_0(z))^{-1} \end{array}\right)
\end{equation}
It is clear that $\phi^{-1}((\mathcal{E}^{(n-1)P})_P)\subseteq (\mathcal{E}^{nP})_P$, and that this map of stalks is the inverse of the map of stalks induced by $\phi$.
\end{proof}
Note that condition (\ref{HBC1}) (and hence conditions (\ref{HBC2}) and (\ref{HBC3})) is invariant under twist: if $\phi^E$ is a section of $\mathrm{End}(E)$ satisfying (\ref{HBC1}) near a parabolic point, and $F$ is a twist of $E$, then the induced section $\phi^F$ of $\mathrm{End}(F)$ also satisfies condition (\ref{HBC1}).

Now we are ready to state the main definition of this section.
\begin{definition}\label{def:cylinder Higgs bundle}
Let $E\to\CP^1$ be a rank $k$ framed parabolic vector bundle with two parabolic points such that $\pardeg(E)=0$.  Suppose that the parabolic structure at each parabolic point $P$ is full and that the weights satisfy
\begin{equation}
\label{PHB weights definition}
\alpha_P^{i}-\alpha_P^{i-1} = \frac{1}{k}\quad\text{for}\quad 0<i<k.
\end{equation}
Let $\phi$ be a meromorphic section of $\End(E)$ which is holomorphic away from the parabolic points and which satisfies any of the three equivalent conditions of proposition \ref{prop:HBC} at each parabolic point.  Then the equivalence class of a pair $(E,\phi)$ under bundle isomorphisms is called a \emph{framed cylinder Higgs bundle}.  An equivalence class of framed cylinder Higgs bundles under twist is called an \emph{unframed cylinder Higgs bundle}.
\end{definition}

The usual approach to Higgs bundles involves a section of $\End(E)\otimes K$, rather than $\End(E)$.  One can easily obtain such a section from our definition by tensoring $\Phi$ with a section of $K_{\CP^1}$ with simple poles at the parabolic points.  The resulting section $\Phi$ will have poles of order 2 at the parabolic points, and the pair $(E,\Phi)$ is an example of a \emph{wild} Higgs bundle.  In the terminology of \cite{mochizuki2011}, it is an example of a \emph{good filtered Higgs bundle} (but not an unramifiedly good filtered Higgs bundle).  From a different perspective, $(E,\phi)$ could be thought of as a \emph{twisted Higgs bundle} (see e.g.\ \cite{bnr,markman}) with parabolic structures (see e.g.\ \cite{konno}).


Before proceeding to describe the relationship between cylinder Higgs bundles and monopole chains, let us fix some conventions.  By choice of coordinate $w$ we identify $\CP^1$ with $\CC\cup\{\infty\}$, and without loss of generality we assume that the two parabolic points are $w=0$ and $w=\infty$.  The boundary condition \eqref{HBC3} implies that $\det\phi$ has simple poles at each parabolic point, so that
\begin{equation}
\label{eq:det phi}
(-1)^k\det\phi = a_0-\frac{c_{0}}{w}-c_\infty w
\end{equation}
for constants $c_0, a_0, c_\infty$.  We may assume without loss of generality that $w$ is chosen such that $c_0=c_\infty=:c$.

Let us fix $\beta>0$.  We obtain from a cylinder Higgs bundle $(E,\phi)$ a Higgs bundle (in the usual sense) $(E|_{\CC^\ast},\Phi)$ over $\CC^\ast=\CP^1\setminus\{0,\infty\}$ by setting $\Phi=\phi\frac{\dd w}{2\beta w}$.  A hermitian metric $h$ on $E$ which is compatible with the parabolic structure is called \emph{hermitian-Einstein} if
\begin{equation}
\label{eq:HE}
\mathcal{F}^h:=F^h+[\Phi,\Phi^{\ast h}] = 0.
\end{equation}
Here $F^h$ denotes the curvature of the Chern connection $A^h$ of $h$ and $\Phi^{\ast h}$ denotes the hermitian conjugate with respect to $h$.  The commutator of 1-forms is understood in a graded sense: $[\Phi,\Phi^{\ast h}]:=\Phi\wedge\Phi^{\ast h}+\Phi^{\ast h}\wedge\Phi$.  It is sometimes convenient to introduce operators $D''=\bar{\pa}+\Phi$ and $D'_h = \pa^h+\Phi^{\ast h}$; then the quantity $\mathcal{F}^h$ defined in \eqref{eq:HE} is equal to $D''D'_h+D'_hD''$, and can be understood as the curvature of the connection $D''+D'_h$.


Given existence of a hermitian-Einstein metric, the hermitian-Einstein equation \eqref{eq:HE} and the condition that $\phi$ is holomorphic may be rewritten in a unitary gauge with respect to the real coordinates $x^1,x^2$ defined by $w=\exp(\beta(x^1+\ii x^2))$ as:
\begin{align}
\label{eq:Hit1}
F_{12} &= \frac{\ii}{2}[\phi,\phi^\dagger] \\
\label{eq:Hit2}
0 &=  \frac{\pa\phi}{\pa x^1} + \ii\frac{\pa\phi}{\pa x^2} + [A_1+\ii A_2,\phi].
\end{align}

These equations are known as Hitchin's equations.  Cherkis and Kapustin \cite{ck} established a bijection between monopole chains and solutions of Hitchin's equations on the cylinder subject to the following conditions:
\begin{enumerate}
\item[H1.] $\Tr(\phi(w)^\alpha)$ for $\alpha=1,\ldots,k-1$ extend to holomorphic functions on $\CP^1$;
\item[H2.] $\det(\phi(w))$ extends to a meromorphic function on $\CP^1$ with simple poles at $0$ and $\infty$;
\item[H3.] $|F_{12}|^2 = O(|x^1|^{-3})$ as $|x^1|\to\infty$.
\end{enumerate}

Solutions of Hitchin's equations satisfying H1, H2, H3 correspond to cylinder Higgs bundles.  More precisely,
\begin{theorem}\label{thm:KH}
\begin{itemize}
\item[(i)] Every unframed cylinder Higgs bundle admits a compatible hermitian-Einstein metric which is  unique up to scaling.
\item[(ii)] The norm of the curvature of the Chern connection of the metric constructed in (i), measured using the cylindrical metric $x^i\dd x^i$ on $\RR\times S^1$, decays faster than any exponential function of $|x^1|$ as $|x^1|\to\infty$.  In particular, the corresponding solution of Hitchin's equations satisfies H1, H2 and H3.
\item[(iii)] Conversely, given any solution of Hitchin's equations on the cylinder satisfying H1, H2, H3, the underlying bundle $E$ and section $\phi$ can be extended to a cylinder Higgs bundle on $\CP^1$.
\end{itemize}
\end{theorem}
Part (i) of this theorem follows from the fact that every good filtered Higgs bundle admits a wild harmonic metric.  This is proved in \cite{mochizuki} in a rather general setting.  The corresponding statement for \emph{unramifiedly} good filtered Higgs bundles over Riemann surfaces was proved earlier in \cite{biquardboalch}, and the statement for good filtered Higgs bundles over Riemann surfaces is easily deduced from this by taking a ramified covering.  For completeness, we give a direct proof of part (i) of the theorem in an appendix.

Part (ii) of this theorem follows from proposition 7.2.9 of \cite{mochizuki2011}.  Again, a direct proof of this part is given in the appendix to this paper.

Part (iii) of the theorem follows from theorem 21.3.1 of \cite{mochizuki2011}, which states that every good wild harmonic bundle admits a natural filtration.  According to \cite{mochizuki2011}, in the case of Riemann surfaces the proof is simpler than that given in \cite{mochizuki2011} and follows from earlier work of Simpson.

Combined with Cherkis--Kapustin's results on the Nahm transform, theorem \ref{thm:KH} gives a natural correspondence between monopole chains and cylinder Higgs bundles.

\section{Spectral data}
\label{sec:3}

In the previous section we saw that monopole chains correspond to cylinder Higgs bundles, i.e.\ twisted Higgs bundles over $\CP^1$ with prescribed parabolic structures.  In this section we describe how cylinder Higgs bundles correspond to spectral data consisting of curves in $\CP^1\times\CP^1$ equipped with parabolic line bundles.  Spectral curves are a standard feature of the theory of Higgs bundles \cite{hitchin}, and their relevance to monopole chains was already highlighted in \cite{ck}.  The novel contribution of this section is the incorporation of parabolic structures.

Let $(E,\phi)$ be a framed cylinder Higgs bundle.  The associated spectral curve $S\subset\mathbb{CP}^1\times\mathbb{CP}^1$ is defined by the equation
\begin{equation}
\det(\zeta \mathrm{Id}_E -\phi(w)) = 0.
\end{equation}
More precisely, let $\pi_w,\pi_\zeta:\CP^1\times\CP^1\to\CP^1$ be the projections onto the first and second factors, so that $\pi_w(w,\zeta)=w$, $\pi_\zeta(w,\zeta)=\zeta$ for $w,\zeta\in\CC\subset\CP^1$.  Let $z_0,z_1$ be holomorphic sections of $\mathcal{O}(1)\to\CP^1$ such that $z_0/z_1=\zeta$.  Then $S\subset\CP^1\times\CP^1$ is the vanishing set of the section $\det(\pi_\zeta^\ast z_0\mathrm{Id}_E-\pi_\zeta^\ast z_1\pi_w^\ast \phi)$ of $\pi_\zeta^\ast\mathcal{O}(k)$.

Let us consider the form of the spectral curve in more detail.  Near the point $w=0$ we may choose a trivialisation so that $\phi$ takes the form given in equation \eqref{eq:HBC3} (with local coordinate $z=w$).  Then
\begin{equation}
\det(\zeta\mathrm{Id}_E-\phi) = -\frac{c}{w}+\zeta^k+\sum_{i=0}^{k-1}\zeta^if_i(w).
\end{equation}
Thus the coefficient of $\zeta^i$ for $i>0$ is a holomorphic function of $w$ near $w=0$, while the coefficient of $\zeta^0$ has a simple pole.  A similar analysis near $w=\infty$ shows that the coefficients of $\zeta^i$ for $i>0$ are also holomorphic near $w=\infty$, and the coefficient of $\zeta^0$ again has a simple pole.  It follows that the coefficients of $\zeta^i$ are constant functions of $w$ for $i>0$, while the coefficient of $\zeta^0$ is $-c(w+w^{-1})$ plus a constant \footnote{recall that the coefficients of $w$ and $w^{-1}$ were fixed to be equal by our choice of coordinate $w$ -- see the discussion around equation \eqref{eq:det phi}}.  Thus the equation defining the spectral curve takes the form \cite{ck}
\begin{equation}
\label{spectral curve}
-cw-\frac{c}{w}+\zeta^k+\sum_{i=0}^{k-1}a_i\zeta^i = 0
\end{equation}
for constants $a_i\in\CC$.  Observant readers may recognise equation \eqref{spectral curve} as the spectral curve for Toda mechanics\footnote{I am grateful to S.~Ruisenaars for this observation}.  This reflects the fact that, regarded as an integrable system, the moduli space of Higgs bundles on a cylinder is the Toda model \cite{marshakov}.

The spectral curve $S$ is irreducible and hence an integral scheme.  To see this, write the defining equation as a polynomial in $w$ with coefficients in $\CC[\zeta]$:
\begin{equation}
w^2-c^{-1}w\left(\zeta^k+\sum_{i=0}^{k-1}a_i\zeta^i\right)+1.
\end{equation}
This satisfies Eisenstein's criterion, because there exists a degree one polynomial that divides the coefficient of $w$ but does not divide the coefficients of $w^2$ or $w^0$.  Therefore the polynomial is irreducible, and $S$, which equals the closure in $\CP^1\times\CP^1$ of the associated affine variety, is also irreducible.

In fact, for generic values of the coefficients $a_i$ the spectral curve is nonsingular.  The map $S\to\CP^1$ given by $(w,\zeta)\mapsto\zeta$ is two-to-one, so this curve is hyperelliptic.  Its genus is $k-1$ \cite{ck}.  For all values of $a_i$ the curve contains the points $P_0,P_\infty$ with coordinates $(w,\zeta)=(0,\infty)$ and $(\infty,\infty)$.

In cases where $S$ is a nonsingular curve it carries a natural line bundle $L$.  This is defined to be the cokernel of the map
\begin{equation}
\pi_\zeta^\ast z_0\mathrm{Id}_E-\pi_\zeta^\ast z_1\pi_w^\ast \phi:\pi_w^\ast E\otimes\pi_\zeta^\ast\mathcal{O}(-1)\to \pi_w^\ast E.
\end{equation}
The fibre of $L$ at a point $(w,\zeta)\in S\cap\CC^\ast\times\CC$ is then
\begin{equation}
L_w = E_w/\operatorname{Im}\big(\phi(w)-\zeta\mathrm{Id}:E_w\to E_w\big).
\end{equation}
The sheaf of holomorphic sections of $L$ is denoted $\mathcal{L}$.  If $S$ is singular there is no line bundle but one still has a torsion free invertible sheaf $\mathcal{L}$.

The bundle $E$ and endomorphism $\phi$ can be recovered from the spectral data \cite{hitchin,bnr}.  Let $\pi:S\to\CP^1$ be the $k$-to-one map $\pi(w,\zeta)=w$.  The push-forward $\pi_\ast(\mathcal{L})$ is a locally free and of rank $k$, so defines a vector bundle over $\CP^1$.  Now any holomorphic section of $E$ over an open set $U\subset\CP^1$ determines a section of $\pi^\ast E$ over $\pi^{-1}(U)\subset S$ via pull-back, and hence determines a section of $L$ over $\pi^{-1}(U)$ via the map $E\to L$.  Thus there is a natural map
\begin{equation}
\mathcal{E} \to \pi_\ast(\mathcal{L}).
\end{equation}
This map induces an isomorphism from $E$ to the bundle determined by $\pi_\ast(\mathcal{L})$ \cite{hitchin,bnr}, so $E$ can be recovered from the spectral data.  The endomorphism of $L|_{S\setminus\{P_0,P_\infty\}}$ determined by multiplication with $\zeta$ induces an endomorphism of $\mathcal{E}|_{\CC^\ast}\cong \pi_\ast(\mathcal{L}|_{S\setminus\{P_0,P_\infty\}})$ which agrees with $\phi$, so $\phi$ can also be recovered from the spectral data.

Note that the preceding construction differs from that of \cite{hitchin}, which instead takes $L$ to be the kernel of $L$.  The two approaches are however related, as explained in \cite{bnr}.

Having described the relationship between $(E,\phi)$ and $(S,L)$, we now introduce parabolic structures on $L$ and explain how the parabolic structures on $E$ can be recovered from them.  Our construction is similar to one studied in \cite{ab}.  First we introduce filtrations on $L$ at the points $P=P_0,P_\infty$ by setting
\begin{equation}
\mathcal{L}^{iP} = \mathcal{L}[-iP].
\end{equation}
This induces filtrations on $\pi_\ast\mathcal{L}$ by setting
\begin{equation}
(\pi_\ast\mathcal{L})^{i\pi(P)} := \pi_\ast(\mathcal{L}[-iP]).
\end{equation}
The function $\tilde{\zeta}=1/\zeta$ on $S$ has zeros of order 1 at the points $P=P_0,P_\infty$, so multiplying sections with $\tilde{\zeta}$ induces surjections of stalks,
\begin{equation}
\tilde{\zeta}:\mathcal{L}[-iP]_P\to\mathcal{L}[-(i+1)P]_P
\end{equation}
Therefore we have an isomorphism of stalks,
\begin{equation}
((\pi_\ast\mathcal{L})^{i\pi(P)})_{\pi(P)}=\pi_\ast(\tilde{\zeta}^i\mathcal{L})_{\pi(P)},\quad i\geq 0.
\end{equation}
On the other hand, we know from the boundary conditions of $\phi$ that the stalk $(\mathcal{E}^{i\pi(P)})_{\pi(P)}$ of $\mathcal{E}^{i\pi(P)}$ is equal to the image of $\mathcal{E}_P$ under $\phi^{-i}$.  Since the endomorphism $\phi^{-1}$ corresponds to multiplying with $\tilde{\zeta}$ we conclude that the filtration $\pi_\ast(\mathcal{L}[-iP])$ agrees with the filtration $\mathcal{E}^{i\pi(P)}$.  Thus the filtration of $E$ can be recovered from its spectral data.

We have seen that $\mathcal{L}$ has natural filtrations at $P_0$ and $P_\infty$ corresponding to the filtrations of $\mathcal{E}$.  In order to define parabolic structures at these points we introduce weights
\begin{equation}
\label{PHB weights explicit}
\alpha_P^0 = k \alpha_{\pi(P)}^0,\quad P=P_0,P_\infty.
\end{equation}
The weights of $\mathcal{E}$ can easily be recovered from the weights of $\mathcal{L}$ using the formulae
\begin{equation}
\alpha_{\pi(P)}^i = \frac{\alpha_P^0+i}{k},\quad 0\leq i<k.
\end{equation}
The following proposition shows that this association of parabolic weights on $\mathcal{E}$ and $\mathcal{L}$ is natural:

\begin{proposition}
Let $h$ be a compatible hermitian metric on $L\to S$ and let $\mathcal{E}=\pi_\ast\mathcal{L}$.  For any open set $U\subset\CC^\ast$ and section $\sigma\in\pi_\ast\mathcal{L}|_{\pi^{-1}(U)}$ define a function
\begin{equation}
\pi_\ast h(\sigma,\sigma):U\to\RR,\quad \pi_\ast h(\sigma,\sigma)(z) = \sum_{(w,\zeta)\in\pi^{-1}(\{w\})} h(\sigma,\sigma)(w,\zeta).
\end{equation}
Then $\pi_\ast h$ defines a hermitian metric on $E$ near the parabolic points which is compatible with the parabolic structures.
\end{proposition}
\begin{proof}
Close to (but not at) the parabolic points the map $\pi:S\to\CP^1$ is $k$-to-1, and there is a canonical decomposition
\[ E_w\cong\bigoplus_{(w,\zeta)\in\pi^{-1}(\{w\})}L_{(w,\zeta)}.\]
Then $\pi_\ast h$ is equal to the sum of the hermitian metrics on these summands, and is in particular a hermitian metric.

Now we show that this metric is compatible with the parabolic structure of $\mathcal{E}$.  Let $P=P_0$ or $P_\infty$ be one of the parabolic points of $S$.  Then $\tilde{\zeta}=\zeta^{-1}$ is a local coordinate on $S$ that vanishes at $P$, and $z=w$ or $w^{-1}$ is a local coordinate on a neighbourhood $U$ of $\pi(P)\in\CP^1$ that vanishes at $\pi(P)$.  In these coordinates the projection $\pi$ can be written $\tilde{\zeta}\mapsto z(\tilde{\zeta})$, and we know that $|z|=O(|\tilde{\zeta}|^k)$ and $|\tilde{\zeta}|=O(|z|^{1/k})$ as $\tilde{\zeta}\to 0$.

Let $\sigma\in\pi_\ast\mathcal{L}|_U$ and $\alpha\in\RR$.  Suppose that $h(\sigma,\sigma)(\tilde{\zeta})=O(|\tilde{\zeta}|^{2k\alpha})$ as $\tilde{\zeta}\to 0$.  Then
\[
\pi_\ast h(\sigma,\sigma)(z) = \sum_{\tilde{\zeta}\in\pi^{-1}(\{z\})} h(\sigma,\sigma)(\tilde{\zeta})=O(|\tilde{\zeta}|^{2k\alpha})=O(|z|^{2\alpha}).
\]
Conversely, suppose that $\pi_\ast h(\sigma,\sigma)(z)=O(|z|^{2\alpha})$.  Then
\[
h(\sigma,\sigma)(\tilde{\zeta})\leq \pi_\ast h(\sigma,\sigma)(z(\tilde{\zeta})) = O(|z|^{2\alpha})=O(|\tilde{\zeta}|^{2k\alpha}).
\]
Therefore, for $\alpha^{i-1}_{\pi(P)}<\alpha\leq\alpha^i_{\pi(P)}$,
\begin{align*} \pi_\ast h(\sigma,\sigma)=O(|z|^{2\alpha})&\iff h(\sigma,\sigma)=O(|\tilde{\zeta}|^{2k\alpha}) \\ &\iff \sigma\in \mathcal{L}[-iP]|_{\pi^{-1}(U)}=\mathcal{E}^i|_U. \end{align*}
So $\pi_\ast h$ is compatible with the parabolic structure of $\mathcal{E}$.
\end{proof}

%

We summarise this discussion by stating the precise relationship between cylinder Higgs bundles and their spectral data.
\begin{theorem}\label{thm:spectral}
There is a one-to-one correspondence between the moduli space of framed rank $k$ cylinder Higgs bundles and the moduli space of spectral data $(S,L)$, where $S$ is a curve in $\CP^1$ of the form \eqref{spectral curve} and $L\to S$ is a framed parabolic line bundle (or torsion free invertible sheaf) with parabolic structures at the points $P_0$, $P_\infty$ such that
\begin{equation}
\pardeg(L) = k-1.
\end{equation}
This correspondence respects the actions of $\ZZ\times\ZZ$ given by twisting at the parabolic points, so induces a correspondence between the moduli spaces of unframed cylinder Higgs bundles and unframed spectral data.
\end{theorem}
\begin{proof}
The only part that has not been proved in the preceding discussion is the statement about the parabolic degree of $L$.  By the Grothendieck-Riemann-Roch theorem,
\begin{equation}
\frac{c_1(S)}{2} + c_1(L) = k\frac{c_1(\CP^1)}{2} + c_1(E).
\end{equation}
Now $c_1(L)=\deg(L)$, $c_1(E)=\deg(E)$, $c_1(\CP^1)=2$ and $c_1(S)=2-2g(S)= 4-2k$, so
\begin{equation}
\deg(L) = \deg(E) + 2k-2.
\end{equation}
From the definition of parabolic degree and equation \eqref{PHB weights explicit},
\begin{align}
\pardeg(L) &= \deg(L)+\alpha_{P_0}^0+\alpha_{P_\infty}^0 \\
\pardeg(E) &= \deg(E)+\sum_{i=0}^{k-1}\left(\frac{\alpha_{P_0}^0 + i}{k}+\frac{\alpha_{P_\infty}^0 + i+1-k}{k}\right) \\
&= \deg(E) + \alpha_{P_0}^0+\alpha_{P_\infty}^0 + (k-1).
\end{align}
Since $\pardeg(E)=0$, 
\begin{equation}
\pardeg(L) = \pardeg(E)-(k-1) +2-2k = k-1.
\end{equation}
\end{proof}

We have defined parabolic structures on $L\to S$ using cylinder Higgs bundles.  The spectral curve $S$ and line bundle $L$ can equivalently be defined directly in terms of the monopole chain \cite{ck}.  Mochizuki \cite{mochizuki} has identified parabolic structures associated with monopole chains, and these should induce parabolic structures on $L$.  It would be interesting to know whether the two parabolic structures on $L$ associated with Higgs bundles and monopole chains agree.

\section{Spectral data with cyclic symmetry}
\label{sec:4}

\subsection{Group actions on the moduli space}

In this section we will study fixed points of groups which act naturally on the moduli space of cylinder Higgs bundles.  We begin by describing the action of these groups, and their interpretation for monopole chains.

First, there is a group $\U(1)_R$ which acts by multiplication on $\phi$:
\begin{equation}
e^{\ii\theta}\cdot(E,\phi) = (E,e^{\ii\theta}\phi).
\end{equation}
This action multiplies the determinant of $\phi$ by a phase and in particular maps the parameter $c$ to $e^{\ii k\theta}c$ (see equation \eqref{eq:det phi}).  Cherkis and Kapustin \cite{ck} regard this parameter $c$ (which is $e^{-\beta\mathfrak{v}}$ in their notation) as fixed, and they study moduli spaces of solutions of Hitchin's equations with a fixed value of $c$.  So from their perspective, only the subgroup $\ZZ_{k}$ of $\U(1)_R$ acts on the moduli space, and in general elements of $\U(1)_R$ map from one moduli space to another.  We however do not give the parameter $c$ special status, so $\U(1)_R$ acts on the moduli spaces of framed and unframed cylinder Higgs bundles.

Under this action the spectral curve $S$ maps to its image under the transformation
\begin{equation}
(w,\zeta)\mapsto (w,e^{\ii\theta}\zeta)
\end{equation}
of $\CP^1\times\CP^1$.  The line bundle $L$ maps to its pull-back under the inverse of this map.  For monopole chains on $\RR^2\times S^1_\beta$, the action of $\U(1)_R$ corresponds to rotation of the plane $\RR^2$.

Next, there is a group $\ZZ_2$ which acts on $(E,\phi)$ as pull back under the map $w\mapsto -w$.  Note that under this map the coefficients of $w$ and $1/w$ in our expression \eqref{eq:det phi} for $\det(\phi)$ are multiplied by -1, so the action respects the condition that these coefficients are equal.  The spectral curve $S$ is mapped to its image under the map
\begin{equation}
(w,\zeta) \mapsto (-w,\zeta),
\end{equation}
acting on $\CP^1\times\CP^1$, and the line bundle $L$ is mapped to its image under pull-back.

The corresponding action on monopole chains is to twist with a line bundle over $\RR^2\times S^1_\beta$ equipped with a flat connection whose holonomy about the circle is $-1$.  This is equivalent to what is known in the physics literature as a ``large'' gauge transformation.  The analogous action on the moduli space of calorons (i.e.\ instantons on $\RR^3\times S^1$) is sometimes known as the rotation map \cite{cork,nye}.

Finally, there is an action of $\RR$ on the moduli space of framed cylinder Higgs bundles given by adding real numbers to the parabolic weights:
\begin{equation}
\label{eq:T action}
(\alpha_{\pi(P_0)}^i,\alpha_{\pi(P_\infty)}^i)\mapsto (\alpha_{\pi(P_0)}^i+\chi,\alpha_{\pi(P_\infty)}^i-\chi).
\end{equation}
We remind the reader that the points $\pi(P_0),\pi(P_\infty)\in\CP^1$ are those with coordinates $w=0,\infty$.  This action is defined in such a way that the parabolic degree is unchanged.  If $\chi=n\in\ZZ$ then, up to twist, this action is equivalent to
\begin{equation}
(\mathcal{E},\alpha_{\pi(P_0)}^i,\alpha_{\pi(P_\infty)}^i) \mapsto (\mathcal{E}^{-kn\pi(P_0)+kn\pi(P_\infty)},\alpha_{\pi(P_0)}^i,\alpha_{\pi(P_\infty)}^i)
\end{equation}
because $\alpha_P^i\pm n=\alpha_P^{i\pm kn}$.  Now $\mathcal{E}^{-kn\pi(P_0)+kn\pi(P_\infty)}\cong \mathcal{E}[n\pi(P_0)-n\pi(P_\infty)] \cong \mathcal{E}$ since the line bundle on $\CP^1$ associated to the divisor $\pi(P_0)-\pi(P_\infty)$ is trivial.  Therefore $\ZZ\subset\RR$ acts trivially and there is a well-defined action of $\U(1)_T:=\RR/\ZZ$ on the moduli space of unframed cylinder Higgs bundles.  The corresponding action on spectral data is
\begin{equation}
(S,L,\alpha_{P_0}^0,\alpha_{P_\infty}^0) \mapsto (S,L,\alpha_{P_0}^0+k\chi,\alpha_{P_\infty}^0-k\chi).
\end{equation}

If $h$ is a hermitian-Einstein metric on $E$ compatible with the parabolic weights $(\alpha_{\pi(P_0)}^i,\alpha_{\pi(P_\infty)}^i)$ then the unique hermitian-Einstein metric compatible with the weights in equation \eqref{eq:T action} is $|w|^{2\chi}h$.  The Chern connection of $|w|^{2\chi}h$ differs from that of $h$ by $\chi\dd w/w = \chi\beta(\dd x^1+\ii\dd x^2)$.  Transforming to a unitary gauge, we see that the action of $\U(1)_T$ on solutions of Hitchin's equations is
\begin{equation}
(A_1,A_2,\phi)\mapsto (A_1,A_2+\ii\chi\beta\mathrm{Id}_k,\phi).
\end{equation}
For monopole chains on $\RR^2\times S^1_\beta$, this corresponds to translation in the circle $S^1_\beta=\RR/\beta\ZZ$ by $\beta\chi$.

\subsection{Maximal symmetry}

We have seen that there is a natural action of the group
\begin{equation}
\ZZ_2\times\U(1)_R\times\U(1)_T
\end{equation}
on the moduli space of unframed cylinder Higgs bundles.  We now seek points in this moduli space with non-trivial stabiliser.  Our first result is
\begin{lemma}
Let $G\subset\ZZ_2\times\U(1)_R\times\U(1)_T$ be a subgroup which fixes a point in the moduli space of unframed cylinder Higgs bundles of rank $k$.  Then the image of $G$ under the projection $\ZZ_2\times\U(1)_R\times\U(1)_T\to\ZZ_2\times\U(1)_R$ is a subgroup of the cyclic group of order $2k$ generated by
\begin{equation}
\label{R definition}
R:(w,\zeta)\mapsto (-w,\exp(-\ii\pi/k)\zeta).
\end{equation}
If the image of $G$ equals the whole of this cyclic group then the spectral curve $S$ of the fixed point is nonsingular and given by
\begin{equation}
\label{symmetric spectral curve}
cw^2+c-w\zeta^k = 0.
\end{equation}
\end{lemma}
\begin{proof}
Recall that the equation of the spectral curve takes the form $cw^2+c-w\left(\zeta^k+a_{k-1}\zeta^{k-1}+\ldots\right)$.  Written this way, the coefficients of $w^2$ and $w^0$ are unchanged under transformations in $\ZZ_2\times\U(1)_R\times\U(1)_T$, so a transformation fixes the curve if and only if it fixes the remaining coefficients of the remaining terms.

Now $\U(1)_T$ acts trivially on the curve, while $(\pm1,e^{\ii\theta})\in\ZZ_2\times\U(1)_R$ multiplies the coefficient of $w\zeta^j$ by $\pm e^{-\ii j\theta}$.  The coefficient of $w\zeta^k$ is non-zero, and clearly invariant if and only if $(\pm1,e^{\ii\theta})=R^n$ for some $n$.  Finally, the remaining coefficients $a_i$ are invariant under $R$ if and only if they are all zero.
\end{proof}

Motivated by this result, we say that a stabiliser subgroup $G\subset\ZZ_2\times \U(1)_R\times\U(1)_T$ is \emph{maximal} if its image in $\ZZ_2\times \U(1)_R$ is generated by the transformation $R$ defined in \eqref{R definition}.  From the perspective of monopole chains on $\RR^2\times S^1_\beta$, maximal groups are those which act on $\RR^2$ as the cyclic group of order $2k$.

\subsection{Classification}

The goal of the remainder of this section is to classify points in the moduli space of cylinder Higgs bundles with maximal symmetry group.  If a cylinder Higgs bundle has maximal stabiliser group then its spectral curve $S$ is of the form \eqref{symmetric spectral curve} and its parabolic line bundle must be invariant under the lifted action of $R$ up to twist: thus
\begin{equation}
((R^{-1})^\ast\mathcal{L},\alpha_{P_0}^0+k\chi,\alpha_{P_\infty}^0-k\chi) = (\mathcal{L}[-lP_\infty-mP_0],\alpha_{P_0}^0+m,\alpha_{P_\infty}+l)
\end{equation}
for some $\chi\in\RR$ and $l,m\in\ZZ$.  Clearly this equation is solved by choosing $m=-l$, $\chi=-l/k$ and choosing $\mathcal{L}$ such that
\begin{equation}
\label{line bundle fixed}
R^\ast \mathcal{L} = \mathcal{L}[lP_\infty-lP_0].
\end{equation}
Thus in order to classify cylinder Higgs bundles with maximal symmetry we need to classify line bundles $L$ solving equation \eqref{line bundle fixed}.  A cylinder Higgs bundle whose line bundle solves this equation will be invariant under the action of the order $2k$ cyclic subgroup of $\ZZ_2\times\U(1)_R\times\U(1)_T$ generated by
\begin{equation}
(-1,\exp(\pi\ii/k),\exp(2l\pi\ii /k)).
\end{equation}
This group will be denoted $\ZZ_{2k}^{(2l)}$.

In order to solve equation \eqref{line bundle fixed} we employ the Abel-Jacobi map, which gives an explicit parametrisation of line bundles on $S$ of fixed degree.  Given bases $\omega^j$ and $\delta_i$ for $H^0(S,\Omega^{1,0})$ and $H_1(S,\mathbb{Z})$, with $j=1,\dots,g$ and $i=1,\dots,2g$, the period lattice $\Pi\subset \CC^g$ is the lattice generated by the vectors
\begin{equation}
\label{periods}
\Pi_i = \left(\int_{\delta_i}\omega^1,\dots,\int_{\delta_i}\omega^g\right).
\end{equation}
The Jacobian $\mathcal{J}$ is defined to be the quotient $\CC^g/\Pi$.  The Abel-Jacobi map $\mu$ sends degree 0 divisors to points in the Jacobian and is defined by the equations
\begin{equation}
\mu(P-Q) = \left(\int_{Q}^P\omega^1,\dots,\int_{Q}^P\omega^g\right)\quad\mbox{and}\quad \mu(D+D')=\mu(D)+\mu(D').
\end{equation}
A divisor lies in the kernel of this map if and only its induced line bundle is holomorphically trivial, so the Abel-Jacobi map is a bijection from the space of degree 0 line bundles to the Jacobian.

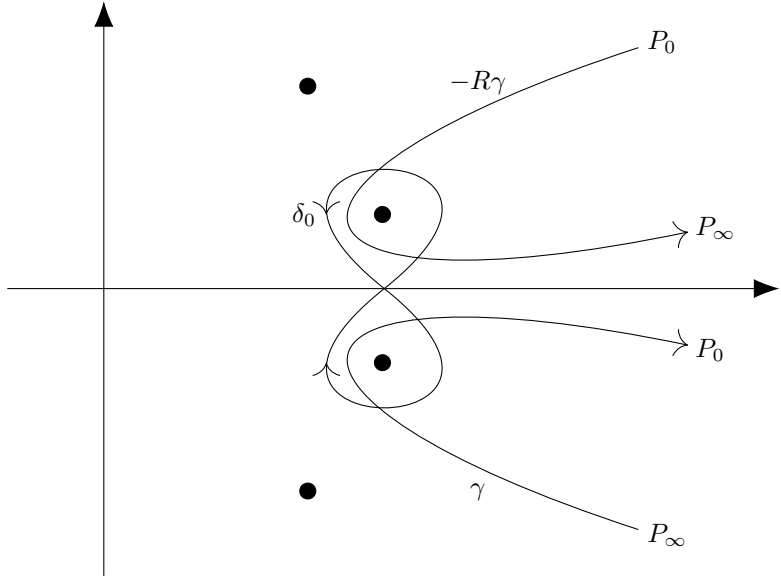
\begin{figure}[h]
\begin{center}
\begin{tikzpicture}
\draw[arrows={-Latex[scale=2.0]}] (-0.5in,0) -- (3.5in,0);
\draw[arrows={-Latex[scale=2.0]}] (0,-1.5in) -- (0,1.5in);
\filldraw[rotate=15,fill=black] (108pt,0pt) circle[radius=3pt];
\filldraw[rotate=45,fill=black] (108pt,0pt) circle[radius=3pt];
\filldraw[rotate=-15,fill=black] (108pt,0pt) circle[radius=3pt];
\filldraw[rotate=-45,fill=black] (108pt,0pt) circle[radius=3pt];
\draw[arrows={<[scale=2.0]-},rotate=15] (3in,-36pt) .. controls (54pt,-27pt) and (54pt, 27pt) .. (3in,36pt);
\draw[arrows={->[scale=2.0]},rotate=-15] (3in,-36pt) .. controls (54pt,-27pt) and (54pt, 27pt) .. (3in,36pt);
\draw (105pt,0pt) .. controls (30pt,60pt) and (180pt,60pt) .. (105pt,0pt);
\draw (105pt,0pt) .. controls (30pt,-60pt) and (180pt,-60pt) .. (105pt,0pt);
\draw[arrows={<[scale=2.0]-}] (83.3pt,28pt) -- (83.3pt,29pt);
\draw[arrows={<[scale=2.0]-}] (83.3pt,-28pt) -- (83.3pt,-29pt);
\draw (200pt,-93pt) node[right]{$P_\infty$};
\draw (200pt,93pt) node[right]{$P_0$};
\draw (218pt,-23pt) node[right]{$P_0$};
\draw (218pt,23pt) node[right]{$P_\infty$};
\draw (140pt,-70pt) node[below]{$\gamma$};
\draw (140pt,70pt) node[above]{$-R\gamma$};
\draw (83pt,28pt) node[left]{$\delta_0$};
\end{tikzpicture}
\end{center}
\caption{Curves used to construct the homology basis for $S$ (see text for details).  The curve $S$ is represented as a two-sheeted branched covering over $\CC$ using the map $S\ni(w,\zeta)\mapsto \zeta$, and the branch points are indicated by solid circles.}\label{sketch}
\end{figure}

Now we describe convenient choices of basis for the homology and cohomology groups of the curve $S$ defined by \eqref{symmetric spectral curve}.  It is useful to regard $S$ as a branched double cover over $\CC$, with covering map $(w,\zeta)\mapsto \zeta$ and branch points the roots of $\zeta^k=2c$.  Let $\gamma$ denote the curve in $ S$ which starts at 
$P_\infty$ and ends at $P_0$ and whose image in the $\zeta$-plane encloses one branch point, as depicted in figure \ref{sketch}.  We denote by $-R\gamma$ the image of $\gamma$ under the action $R$ with reversed orientation, and by $\gamma-R\gamma$ the closed curve obtained by joining $\gamma$ and $-R\gamma$.  Then $\gamma-R\gamma$ is homologous in $H_1(S,\ZZ)$ to the curve $\delta_0$ depicted in figure \ref{sketch}.  The $2k$ images $\delta_i=R^i\delta_0$ of this curve under the action of $\mathbb{Z}_{2k}$ provide a spanning set for $H_1(S,\ZZ)$ (but they are not linearly independent).

A convenient basis for $H^0( S,\Omega^{1,0})$ is
\begin{equation}
 \omega^j = C_j \zeta^{j-1}\frac{\dd\zeta}{w-\zeta^k/2c},\quad j=1,\ldots,k,
\end{equation}
with $C_j\in\CC$ denoting some normalisation constants which are yet to be chosen (see page 255 of \cite{gh}).  This cohomology basis transforms nicely under the action of $R^{-1}$ defined in \eqref{R definition}:
\begin{equation}
\label{R on omega}
 R^\ast\omega^j = -e^{-j\pi\ii/k}\omega^j.
\end{equation}
It follows that the integrals of $\omega^j$ over the curves $R^i\gamma$ and $\delta_i$ are determined by its integral over $\gamma$:
\begin{align}
\int_{R^i\gamma} \omega^j &= \int_{\gamma} (R^i)^\ast\omega^j \\
&= (- e^{-j\pi\ii/k})^i \int_\gamma \omega^j \\
\int_{\delta_i}\omega^j &= \int_{R^i\gamma} \omega^j - \int_{R^{i+1}\gamma}\omega^j \\
&= \big((- e^{-j\pi\ii/k})^i-(- e^{-j\pi\ii/k})^{i+1}\big)\int_\gamma\omega^j.
\end{align}
We will choose the constants $C_j$ so that
\begin{equation}
\int_\gamma\omega^j=1.
\end{equation}
We remark that the integral on the left of this equation is guaranteed to be non-zero, since if it was not, the integral of $\omega^j$ over all the curves $\delta_i$ would be zero, contradicting the fact that $\omega^j$ represents a nontrivial class in $H^0(S,\Omega^{1,0})$ and $\delta_i$ span $H_1(S,\ZZ)$.

With these choices of bases, the generators of the period lattice defined in \eqref{periods} are
\begin{equation}
\label{Pi explicit}
\Pi_i = \rho^i(1-\rho) \left( \begin{array}{c} 1 \\ 1 \\ \vdots \\ 1 \end{array}\right), \quad i=0,\dots,2k-1,
\end{equation}
where $\rho$ is the matrix representing the action of $R$:
\begin{equation}
\rho = -{\rm diag}(e^{-\pi\ii/k},e^{-2\pi\ii/k},\ldots,e^{-(k-1)\pi\ii/k}).
\end{equation}

The divisor 
$(P_0-P_\infty)$ that appears in equation \eqref{line bundle fixed} corresponds under the Abel-Jacobi map to the point in the Jacobian represented by the following vector:
\begin{equation}
\Gamma_0 = \left(\begin{array}{c} \int_\gamma \omega^1 \\ \int_\gamma \omega^2 \\ \vdots \\ \int_\gamma \omega^{k-1} \end{array} \right) = \left(\begin{array}{c} 1 \\ 1 \\ \vdots \\ 1 \end{array}\right).
\end{equation}
Addition of this vector generates an action of $\ZZ$ on the Jacobian, and the quotient by this action will be denoted $\mathcal{J}/\langle \Gamma_0\rangle$.  By equation \eqref{Pi explicit} $\rho\Gamma_0-\Gamma_0\in\Pi$, so multiplication with the matrix $\rho$ gives a well-defined action on $\mathcal{J}/\langle \Gamma_0\rangle$.  Any solution $(\mathcal{L},l)$ of \eqref{line bundle fixed} determines a fixed point of $\rho$ in $\mathcal{J}/\langle \Gamma_0\rangle$, so finding these fixed points helps to solve equation \eqref{line bundle fixed}.  The fixed points are classified in the following:
\begin{proposition}
There are precisely $k$ fixed points of $\rho$ in $\mathcal{J}/\langle \Gamma_0\rangle$.  They are represented by the vectors,
\begin{equation}
l(1-\rho)^{-1} \Gamma_0,\quad l=0,1,\ldots,k-1.
\end{equation}
\end{proposition}
\begin{proof}
We begin with a few linear algebraic observations.  Let
\begin{equation}
\Gamma_i = \left(\begin{array}{c} \int_{R^i\gamma} \omega^1 \\ \int_{R^i\gamma} \omega^2 \\ \vdots \\ \int_{R^i\gamma} \omega^{k-1} \end{array} \right) = \rho^i \left(\begin{array}{c} 1 \\ 1 \\ \vdots \\ 1 \end{array}\right),\quad i=0,\ldots 2k-1.
\end{equation}
The vectors $\Gamma_i$ and $\Pi_i$ are related as follows:
\begin{equation}
\label{Pi and P}
\Pi_i  = \Gamma_i-\Gamma_{i+1} = (1-\rho)\Gamma_i.
\end{equation}

The $2k$ vectors $\Gamma_0,\ldots,\Gamma_{2k-1}$ are not linearly independent over $\RR$: they satisfy two non-trivial relations,
\begin{equation}
\label{sum P2i}
\sum_{i=0}^{k-1}\Gamma_{2i} = 0\quad\text{and}\quad \sum_{i=0}^{k-1}\Gamma_{2i+1} = 0,
\end{equation}
because roots of unity sum to zero.  However, the vectors $\Gamma_0,\ldots,\Gamma_{2k-3}$ are linearly independent over $\RR$.  To show this, suppose that $\sum_{i=0}^{2k-3}\alpha_i\Gamma_i=0$ for some real numbers $\alpha_i$.  Then additionally $\sum_{i=0}^{2k-3}\alpha_i\bar{\Gamma_i}=0$, so
\begin{equation}
\sum_{i=0}^{2k-3} \alpha_i(-e^{-\pi \ii j/k})^i = 0\quad\text{and}\quad \sum_{i=0}^{2k-3} \alpha_i(-e^{\pi \ii j/k})^i = 0 \quad\forall j=1,\ldots,k-1,
\end{equation}
in other words, the polynomial $\sum_{i=0}^{2k-3} \alpha_i x^i$ has $2k-2$ distinct roots.  Since the degree of this polynomial is less than or equal to $2k-3$ it must be zero, so the coefficients $\alpha_i$ must vanish.

Since the vectors $\Gamma_0,\ldots,\Gamma_{2k-3}$ are linearly independent and $(1-\rho)$ is invertible, by equation \eqref{Pi and P} the vectors $\Pi_0,\ldots,\Pi_{2k-3}$ are linearly independent and by equations \eqref{Pi and P} and \eqref{sum P2i} they generate $\Pi$.

Let $\Gamma$ be the lattice generated by $\Pi_0,\ldots,\Pi_{2k-3}$ and $\Gamma_0$.  Then $\mathcal{J}/\langle \Gamma_0\rangle=\CC^{k-1}/\Gamma$.  Moreover, $\Gamma$ equals the lattice generated by $\Gamma_0,\ldots,\Gamma_{2k-3}$, because by \eqref{Pi and P} the $\Pi_i$ can be expressed in terms of the $\Gamma_i$, and conversely $\Gamma_i = \Gamma_0-\sum_{j=0}^{i-1}\Pi_j$.

The fixed points of $\rho$ in $\CC^{k-1}/\Gamma$ are represented by solutions $z\in\CC^{k-1}$ of
\begin{equation}
\rho z = z \mod \Gamma.
\end{equation}
The set of solutions $z\in\CC^{k-1}$ to this equation is the lattice $(1-\rho)^{-1}\Gamma$, so the set of solutions in $\CC^{k-1}/\Gamma$ is the quotient $((1-\rho)^{-1}\Gamma)/\Gamma$; we need to determine the size of this set.  The quotient $((1-\rho)^{-1}\Gamma)/\Gamma$ is an abelian group which is clearly isomorphic to $\Gamma/(1-\rho)\Gamma$.  Note that the lattices $(1-\rho)\Gamma$ and $\Pi$ are equal, since by equation \eqref{Pi and P} $\Pi_i = (1-\rho)\Gamma_i$.  Thus the number of fixed points equals the size of the abelian group $\Gamma/\Pi$.

Equation \eqref{Pi and P} implies that 
\begin{equation}
\label{P_i and P_{i+1}}
\Gamma_i=\Gamma_{i+1}\mod\Pi,
\end{equation}
and hence that the group $\Gamma/\Pi$ is generated by $\Gamma_0$.  Equations \eqref{P_i and P_{i+1}} and \eqref{sum P2i} imply that $k\Gamma_0=0\mod\Pi$, so the order of $\Gamma_0$ in $\Gamma/\Pi$ divides $k$.  We claim that the order of $\Gamma_0$ equals $k$.  It will follow that $\Gamma/\Pi$ has size $k$, and that there are $k$ fixed points of the form stated in the proposition.

To prove our claim we derive an expression for $\Gamma_0$ in terms of the basis $\Pi_0,\ldots,\Pi_{2k-3}$ with the help of equation \eqref{sum P2i}:
\begin{eqnarray}
\Gamma_0 &=& \frac{1}{k}\sum_{i=0}^{k-1} \Gamma_{2i} + \sum_{i=0}^{k-2} \frac{k-1-i}{k}(\Pi_{2i}+\Pi_{2i+1}) \\
&=& \sum_{i=0}^{k-2} \frac{k-1-i}{k}(\Pi_{2i}+\Pi_{2i+1}).
\end{eqnarray}
It is clear from this expression that $l\Gamma_0$ can be written as a linear combination of the $\Pi_i$ with integer coefficients if and only if $l=0\mod k$.  Since the basis vectors $\Pi_{0},\ldots,\Pi_{2k-3}$ generate $\Pi$, the order of $\Gamma_0$ in $\CC^{k-1}/\Pi$ is $k$ as claimed.
\end{proof}

The preceding proposition allows us to prove:
\begin{theorem}\label{thm:cyclic spectral}
For fixed $|c|>0$ and $k\in\NN$ there are, up to the action of $\ZZ_2\times \U(1)_R\times\U(1)_T$, precisely $k$ distinct unframed cylinder Higgs bundles with maximal symmetry.  They are fixed by the groups $\ZZ_{2k}^{(2l)}$, with $l=0,\ldots,k-1$.
\end{theorem}
\begin{proof}
The spectral curve $S$ of a maximally symmetric cylinder Higgs bundle must be the one given in equation \eqref{symmetric spectral curve}.  The group $\U(1)_R$ alters the phase of the parameter $c$, so for fixed $|c|$ and up to the action of $\U(1)_R$ this curve is unique.

By twisting we can arrange that the line bundle $L$ has degree zero, so must be one of the line bundles identified in the previous proposition.  Any such line bundle corresponds to a point in the Jacobian represented by a vector
\begin{equation}
j\Gamma_0 + l(1-\rho)^{-1}\Gamma_0,\quad j\in\ZZ,\, l=0,\ldots,k-1.
\end{equation}
Recall that twisting with the divisor $P_\infty-P_0$ corresponds to adding $\Gamma_0$ in the Jacobian.  Thus by twisting we may arrange that $j=0$, and up to equivalence there are precisely $k$ possibilities for the line bundle $L$, labelled by $l$.

Since $\deg L=0$ and $\pardeg L=k-1$, the parabolic weights must satisfy $\alpha_{P_0}^0+\alpha_{P_\infty}^0=k-1$.  It follows that the weights are unique up to the action of $\U(1)_T$.

Finally, since
\begin{equation}
\rho l(1-\rho)^{-1}\Gamma_0 = l(1-\rho)^{-1}\Gamma_0 -l\Gamma_0,
\end{equation}
$R^\ast \mathcal{L}=\mathcal{L}[lP_\infty-lP_0]$.  Therefore, from the discussion surrounding equation \eqref{line bundle fixed} the cylinder Higgs bundle is invariant under the action of $\ZZ_{2k}^{(2l)}$.
\end{proof}

As an immediate consequence we obtain
\begin{corollary}\label{cor:classification}
For fixed $k\in\NN$, $l=0,\ldots,k-1$ and $|c|>0$ there exists a monopole chain invariant under the action of $\ZZ_{2k}^{(2l)}$.  This monopole chain is unique up to the action of $\ZZ_2\times U(1)_T \times U(1)_R$.  There are no monopole chains invariant under the action of $\ZZ_{2k}^{(2l+1)}$.
\end{corollary}

\section{Construction of monopole chains}
\label{sec:5}

In this section we present pictures of the $\ZZ_{2k}$-symmetric monopole chains whose spectral data we have just discussed.  The construction of the monopole chains proceeds in three stages: (i) find the associated Higgs bundle; (ii) solve the hermitian-Einstein equation to find a solution of Hitchin's equations; (iii) apply the Nahm transform to construct the monopole chain.  Only the first stage can be accomplished explicitly; the remaining two are implemented numerically.  We first describe how the monopole chains are constructed, and finish with a discussion of the qualitative features of these monopole chains.


\subsection{Higgs bundle and metric with cyclic symmetry}

We identify $\CC^\ast$ with $\CC/(2\pi\ii/\beta)\ZZ$ by writing $w=e^{\beta s}$.  We work in a holomorphic trivialisation over the covering space $\CC$, so that the Higgs field $\phi$ is represented by a rank $k$ holomorphic matrix-valued function of $s\in\CC$.  Since this defines a section of a bundle over $\CC/(2\pi\ii/\beta)\ZZ$ it must be periodic up to a holomorphic gauge transformation $U$:
\begin{equation}\label{phi periodic}
\phi(s+2\pi\ii/\beta) = U(s)^{-1}\phi(s)U(s)
\end{equation}
In order to have $\ZZ_{2k}$ symmetry it must satisfy
\begin{equation}\label{phi invariant}
\omega\phi(s+\pi\ii/\beta) = V(s)^{-1}\phi(s)V(s)
\end{equation}
where
\begin{equation}
\omega := \exp(\pi\ii/k)
\end{equation}
and $V$ is a holomorphic invertible matrix-valued function.  Like $\phi$, this function $V$ represents a section of a bundle over $\CC/(2\pi\ii/\beta)$ so must be periodic up to a gauge transformation in the following sense:
\begin{equation}\label{V periodic}
V(s+2\pi\ii/\beta) = U(s)^{-1}V(s)U(s+\pi\ii/\beta).
\end{equation}
Note that the argument of $U$ on the right is $s+\pi\ii/\beta$ rather than $s$ for consistency with equations \eqref{phi periodic} and \eqref{phi invariant}: $U(s)$ represents a map from the fibre $E_{s+2\pi\ii/\beta}$ at $s+2\pi\ii/\beta$ to $E_s$, and $V(s)$ represents a map from $E_{s+\pi\ii/\beta}$ to $E_s$, so both sides of equation \eqref{V periodic} compose nicely to give maps from $E_{s+3\pi\ii/\beta}$ to $E_{s+2\pi\ii/\beta}$.

Thus to construct a $\ZZ_{2k}$-invariant Higgs bundle we need to solve equations \eqref{phi periodic}, \eqref{phi invariant} and \eqref{V periodic}.  Combining equations \eqref{phi invariant} and \eqref{phi periodic}, we see that
\begin{equation}
\label{def:W}
\omega^2\phi(s)=W(s)^{-1}\phi(s)W(s),\text{ where }W(s):=V(s)V(s+\pi\ii/\beta)U(s)^{-1}.
\end{equation}
This equation places constraints on $W$: since $\det(\phi)\neq0$, $\omega^2$ must be an eigenvalue of the adjoint action of $W^{-1}$ with multiplicity at least $k$.  This means that the eigenvalues of $W$ must be distinct and their ratios must be $k$th roots of unity.  We may therefore choose to work in a gauge in which
\begin{equation}
\label{soln:W}
W(s) = w(s)\diag(1,\omega^2,\omega^4,\ldots,\omega^{2k-2})
\end{equation}
for some function $w(s)$.  It follows that
\begin{equation}
\label{soln:phi}
\phi(s) = c^{1/k}\Sh^{-1}\diag(\phi_0(s),\phi_1(s),\ldots,\phi_{k-1}(s))
\end{equation}
where $\Sh$ is 	 the ``shift'' matrix introduced in equation \eqref{def:shift}.

The gauge transformation $V(s)$ in equation \eqref{phi invariant} needs to preserve this form of $\phi$.  Therefore it must be of the form
\begin{equation}
\label{soln:V}
V(s) = \Sh^l\diag(V_0(s),V_1(s),\ldots,V_{k-1}(s))
\end{equation}
for some $l\in\{0,1,\ldots,k-1\}$.  
Then equation \eqref{phi invariant} becomes
\begin{equation}
\label{phi invariant 2}
\omega \phi_j(s+\pi\ii/\beta) = V_{j}(s)V_{j-1}(s)^{-1}\phi_{j+l}(s),\quad j=0,\ldots,k-1
\end{equation}
In this equation and those that follow, indices are to be understood modulo $k$: thus for example $V_{-1}=V_{k-1}$.

From the equation \eqref{eq:det phi} we see that
\begin{equation}
\label{product phij}
e^{-\beta s}+e^{\beta s} = c^{-1}(-1)^{k-1}\det\phi = \prod_{j=0}^{k-1}\phi_j(s).
\end{equation}
The zeros of the function on the left of this equation are the points $s_p=(2p-1)\pi\ii/2\beta$ for $p\in\ZZ$.  Each of these must be a zero of precisely one of the functions $\phi_j$.  By applying a gauge transformation if necessary, we may assume without loss of generality that $-\pi\ii/2\beta $ is a zero of $\phi_0$.  Then equation \eqref{phi invariant 2} tells us that $(2p-1)\pi\ii/2\beta$ is a zero of $\phi_{-lp}$.  Thus the distribution of zeros amongst the functions $\phi_j$ is completely determined by equation \eqref{phi invariant 2}.

This observation guides the choice of the functions $\phi_j$.  Let
\begin{equation}
\label{soln:phij}
\phi_j(s) = \prod_{\begin{array}{c}\scriptstyle{i\in\ZZ_k} \\ \scriptstyle{il=j\mod k}\end{array}}\mu_i(s),\qquad \mu_j(s)=e^{-\beta s/k}-\omega^{2j+1}e^{\beta s/k}.
\end{equation}
These satisfy equation \eqref{product phij} and have zeros in the desired places.  Moreover,
\begin{equation}
\mu_j(s+\pi\ii/\beta)=\omega^{-1}\mu_{j+1}(s)
\end{equation}
and
\begin{equation}
\phi_j(s+\pi\ii/\beta)=\begin{cases}\omega^{-m}\phi_{j+l}(s) & j=0\mod m \\
\phi_{j+l}(s) & j\neq0 \mod m \end{cases},\label{nu_j periodic}
\end{equation}
where
\begin{equation}
m:= \gcd(k,l).
\end{equation}
Inserting these into equation \eqref{phi invariant 2} gives
\begin{equation}
V_{j}(s)V_{j-1}(s)^{-1} = \begin{cases}\omega^{1-m} & j=0\mod m\\ \omega & j\neq 0\mod m\end{cases}.
\end{equation}
The solution of this equation is
\begin{equation}
\label{soln:Vj}
V_j(s) = v(s) \omega^{j\mod m}
\end{equation}
for some function $v(s)$.

The gauge transformation $U$ is determined by $V$ and $W$ via equation \eqref{def:W}:
\begin{align}
U(s) &= W(s)^{-1}V(s)V(s+\pi\ii/\beta) \\
\label{soln:U}&= \Sh^{2l}\diag(U_0(s),\ldots,U_{k-1}(s))\\
\label{soln:Uj}U_j(s)&:= v(s)v(s+\pi\ii/\beta)w(s)^{-1}\omega^{-4l}\omega^{-2m\operatorname{floor}(j/m)},
\end{align}
where $\operatorname{floor} (j/m)=(j-(j\mod m))/m$ denotes the greatest integer less than or equal to $j/m$.

We have now solved equations \eqref{phi periodic} and \eqref{phi invariant}: $\phi$ is given in equations \eqref{soln:phi} and \eqref{soln:phij}, $V$ is given in equations \eqref{soln:V} and \eqref{soln:Vj}, and $U$ is given in equations \eqref{soln:U} and \eqref{soln:Uj}.  It remainds to solve equation \eqref{V periodic}.  A short calculation using equations \eqref{soln:V}, \eqref{soln:Vj}, \eqref{soln:U}, \eqref{soln:Uj} shows that
\begin{equation}
U(s)^{-1}V(s)U(s+\pi\ii/\beta) = \frac{w(s)v(s+2\pi\ii/\beta)}{w(s+\pi\ii/\beta)v(s)}\omega^{2l}V(s).
\end{equation}
Therefore the remaining equation \eqref{V periodic} is equivalent to
\begin{equation}\label{w periodic}
w(s+\pi\ii/\beta)=\omega^{2l}w(s).
\end{equation}

In order to determine the symmetry type of this solution we now turn our attention to the hermitian metric.  This is represented by a positive hermitian matrix-valued function $h$ satisfying
\begin{align}
\label{h periodic} h(s+2\pi\ii/\beta) &= U(s)^\dagger h(s) U(s) \\
\label{h invariant} |e^{\beta l's/k}|^2h(s+\pi\ii/\beta) &= V(s)^\dagger h(s) V(s).
\end{align}
The first of these says that $h$ defines a metric on the bundle over $\CC/(2\pi\ii/\beta)\ZZ$, and the second says that $h$ is invariant under the action of $\ZZ_{2k}^{(2l')}$, where $l'$ is to be determined.  The equations together imply that
\begin{equation}
\label{h and W}
|e^{2\beta l's/k}|^2h(s) = W(s)^\dagger h(s) W(s),
\end{equation}
where $W$ was defined in equations \eqref{def:W} and \eqref{soln:W}.  Taking determinants, we see that
\begin{equation}
|e^{2\beta l's}|^2 = |w(s)^k|^2.
\end{equation}
It follows that $w(s)$ equals $e^{2\beta l' s/k}$ times a phase.  Comparing with \eqref{w periodic} we see that $l'=l$, so the solution has $\ZZ_{2k}^{(2l)}$ symmetry.  Equation \eqref{h and W} is then solved by
\begin{equation}
\label{soln:h}
h(s) = \exp\diag(\psi_0(s),\psi_1(s),\ldots,\psi_{k-1}(s))
\end{equation}
for real functions $\psi_j(s)$.

We now impose the condition that $\det(h)=1$, or equivalently,
\begin{equation}
\label{sum psij}
\sum_{j=0}^{k-1}\psi_j(s) = 0.
\end{equation}
No generality is lost in doing so: the hermitian-Einstein equation implies that $\ln \det h$ is a harmonic, and therefore equal to the real part of a holomorphic function $f$.  Applying the gauge transformation $\exp(-f(s)/2k)\mathrm{Id}_k$ then ensures that $\det h=1$.

With condition \eqref{sum psij} imposed, and $V$ given in equations \eqref{soln:V} and \eqref{soln:Vj}, equation \eqref{h invariant} is solved by
\begin{equation}
\label{psi_j periodic}
v(s)=e^{\beta ls/k},\quad \psi_j(s+\pi\ii/\beta)=\psi_{j+l}(s).
\end{equation}
It is straightforward to check that these conditions also ensure that equation \eqref{h periodic} is solved.

\subsection{Hermitian-Einstein equation and the Nahm transform}

This completes our description of the $\ZZ_{2k}^{(2l)}$-invariant Higgs bundle.  Switching to a unitary gauge, the solution is
\begin{align}
\phi &= c^{1/k}\Sh^{-1}\nonumber\\
&\quad \diag(e^{(\psi_{k-1}-\psi_0)/2}\phi_0,e^{(\psi_0-\psi_1)/2}\phi_1,\ldots,e^{(\psi_{k-2}-\psi_{k-1})/2}\phi_{k-1}) \\
A_1 &= -\frac{\ii}{2}\frac{\pa}{\pa x_2}\diag(\psi_0,\psi_1,\ldots,\psi_{k-1}) \\
A_2 &= \frac{\ii}{2}\frac{\pa}{\pa x_1}\diag(\psi_0,\psi_1,\ldots,\psi_{k-1}),
\end{align}
with $A$ being the Chern connection and $s=x^1+\ii x^2$.  The cases $l=0$ and (for even $k$) $l=k/2$ of this solution were previously obtained by Maldonado \cite{maldonado2}.  Note that, from equations \eqref{nu_j periodic} and \eqref{psi_j periodic}, these are the cases where the functions $\phi_j$ and $\psi_j$ are invariant under $s\mapsto s+2\pi\ii/\beta$.  The remaining cases of the solution are new.

By construction this solves the second of Hitchin's equations \eqref{eq:Hit2}.  Inserting this into the first of Hitchin's equation \eqref{eq:Hit1} (i.e.\ the hermitian-Einstein equation \eqref{eq:HE}) yields
\begin{equation}
\label{psi PDE}
|c|^{-2/k}\triangle\psi_j = |\phi_{j+1}|^2\exp(\psi_j-\psi_{j+1})-|\phi_j|^2\exp(\psi_{j-1}-\psi_j).
\end{equation}
These are the variational equations for the functional
\begin{equation}
\int_0^{2\pi/\beta}\int_{-\infty}^{\infty} \sum_{j=0}^{k-1}\left(\sfrac12|c|^{-2/k}|\nabla\psi_j|^2+|\phi_{j+1}|^2\exp(\psi_j-\psi_{j+1})\right)\dd x^1\dd x^2,
\end{equation}
which is of course the Donaldson-Simpson functional.  Equations \eqref{psi PDE} are a form of the affine Toda field equations.  This is a consequence of invariance under the subgroup $\ZZ_k\subset\ZZ_{2k}$: it has been known for some time that $\ZZ_k$-invariant Higgs bundles are equivalent to affine Toda equations \cite{baraglia}.

The differential equation \eqref{psi PDE} needs to be supplemented with boundary conditions coming from the parabolic structures at $x^1=\pm\infty$.  Rather than deal with the parabolic structures directly, it is more straightforward to proceed by examining the differential equation \eqref{psi PDE}.  From our earlier work we know that the right hand side, which represents $[\phi,\phi^\dagger]$, should tend to zero as $|x^1|\to\infty$.  The asymptotics of $|\phi_j|^2$ are given by
\begin{equation}
|\phi_j|^2\sim \begin{cases} e^{2\beta |x^1|m/k} & j=0\mod m \\
1 & j\neq 0 \mod m.\end{cases}
\end{equation}
Inserting these into the right hand side of equation \eqref{psi PDE} and equating to zero gives
\begin{equation}
\psi_{j+1}-2\psi_j+\psi_{j-1} = \begin{cases} - \frac{2\beta|x^1|m}{k} & j=0\mod m \\ \frac{2\beta|x^1|m}{k} & j=-1\mod m \\
0 & \text{otherwise}.
\end{cases}
\end{equation}
This difference equation has a unique solution satisfying the constraint \eqref{sum psij}, leading to the asymptotic boundary conditions
\begin{equation}
\label{psi asymptotics}
\psi_j \sim \frac{2\beta|x^1|}{k}\left(\frac{m-1}{2} - (j\mod m) \right)\text{ as }|x^1|\to\infty.
\end{equation}

Thus Hitchin data for cyclic monopole chains can be constructed by solving the differential equation \eqref{psi PDE}, subject to the boundary conditions \eqref{psi asymptotics}.  Although the Hitchin equations are (like all reductions of self-dual Yang-Mills) integrable, we have elected to solve this equation numerically.  We used a heat flow technique to solve a discretised version of the equations \eqref{psi PDE} on a finite cylinder defined by $|x^1|\leq L$ for some $L>0$, with Neumann boundary conditions based on \eqref{psi asymptotics} imposed at $x^1=\pm L$.

Having solved the Hitchin equations, the corresponding monopole chain can be constructed using the Nahm transform.  This involves the operators
\begin{align}
\slashed{D}_{y} &= \begin{pmatrix} (y_1+\ii y_2)\mathrm{Id}_k-\phi & 2(\pa_s+A_s)+y_3 \\ 2(\pa_{\bar s}+A_{\bar s})-y_3 & (y_1-\ii y_2)\mathrm{Id}_k-\phi^\dagger \end{pmatrix},\\
\slashed{D}_{y}^\dagger &= \begin{pmatrix} (y_1-\ii y_2)\mathrm{Id}_k-\phi^\dagger & -2(\pa_s+A_s)-y_3 \\ -2(\pa_{\bar s}+A_{\bar s})+y_3 & (y_1+\ii y_2)\mathrm{Id}_k-\phi \end{pmatrix},
\end{align}
which depend on $y\in\RR_3$ and act on $\CC^{2k}$-valued functions $Z(x_1,x_2)$.  Since in our ansatz the Hitchin data is periodic up to a gauge transformation $U$, these functions $Z$ are required to satisfy
\begin{equation}
Z\left(x^1,x^2+\frac{2\pi}{\beta}\right) = \begin{pmatrix} U(x^1,x^2)^{-1} & 0 \\ 0 & U(x^1,x^2)^{-1}\end{pmatrix} Z(x^1,x^2).
\end{equation}
It is known that the dimensions of the spaces of $L^2$-normalisable solutions of $\slashed{D}_yZ=0$ and $\slashed{D}_y^\dagger Z=0$ are respectively 0 and 2 \cite{ck}.

The Nahm transform is a two-step process; the first step is to find for each $y\in\RR^3$ an $L^2$-orthonormal basis for the kernel of $\slashed{D}_{y}^\dagger$, i.e.\ solutions $Z_1(x;y),Z_2(x;y)$ to $\slashed{D}_y^\dagger Z_a=0$ normalised such that
\begin{equation}
\slashed{D}_{y}^\dagger Z_a=0,\quad \int_0^{2\pi/\beta}\int_{-\infty}^\infty Z_a^\dagger Z_b \dd x^1\dd x^2 = \delta_{ab}.
\end{equation}
In the second step the monopole Higgs field is constructed via
\begin{equation}
\label{NT Higgs field}
\hat{\phi}_{ab}(y) = \ii\int_0^{2\pi/\beta}\int_{-\infty}^\infty x^1\, Z_a^\dagger Z_b\, \dd x^1\dd x^2.
\end{equation}

We carried out this process numerically, using our numerical Hitchin data.  Rather than solve the equation $\slashed{D}_y^\dagger Z=0$ directly, we solved instead the equation $\slashed{D}_y\slashed{D}_y^\dagger Z=0$.  The two equations are equivalent (since $\slashed{D}_y$ has zero-dimensional kernel) but the latter is more amenable to numerical solution because the differential operator involved is positive and of second order.  The formula that we used for $\slashed{D}_y\slashed{D}_y^\dagger$ is
\begin{align}
\slashed{D}_y\slashed{D}_y^\dagger &= \begin{pmatrix} [ \phi,\phi^\dagger] - \triangle^{A,\phi,y} & -2(\pa_s\phi+[A_s,\phi]) \\ -2(\pa_s\phi+[A_s,\phi])^\dagger & [\phi^\dagger,\phi]-\triangle^{A,\phi,y} \end{pmatrix} \\
\triangle^{A,\phi,y} &= \big(\pa_1 + A_1\big)^2 + \big(\pa_2 + A_2+\ii y_3\mathrm{Id}_k\big)^2\nonumber\\
&\quad-\frac12\{\phi-y_1-\ii y_2,\phi^{\dagger}-y_1+\ii y_2\}.
\end{align}
(Note that this identity assumes that $A,\phi$ solve Hitchin's equations.)  The operator $\slashed{D}_y\slashed{D}_{y}^\dagger$ was stored as a sparse matrix and its smallest eigenvalues were computed using an algorithm included in the MATLAB software package.  Our numerical approximation to the basis $Z_1,Z_2$ of the kernel was given by the two eigenvectors corresponding to the smallest eigenvalues.  In practice the two smallest eigenvalues were always very close to zero, in the sense that they were less than 1\% of the third-smallest eigenvalue.  This gives an indication of the reliability of our numerical scheme.

Finally, the gauge-invariant quantity $\|\hat{\phi}\|^2 = \sfrac12\Tr(\hat{\phi}\hat{\phi}^\dagger)$ was evaluated for points $y$ in a rectangular lattice by evaluating the integrals \eqref{NT Higgs field}.  From this, the energy density was calculated using Ward's formula $\mathcal{E}=\triangle \|\hat{\phi}\|^2$ \cite{ward81} for the energy density.

\begin{figure}[t]
\begin{center}
\includegraphics[height=110pt]{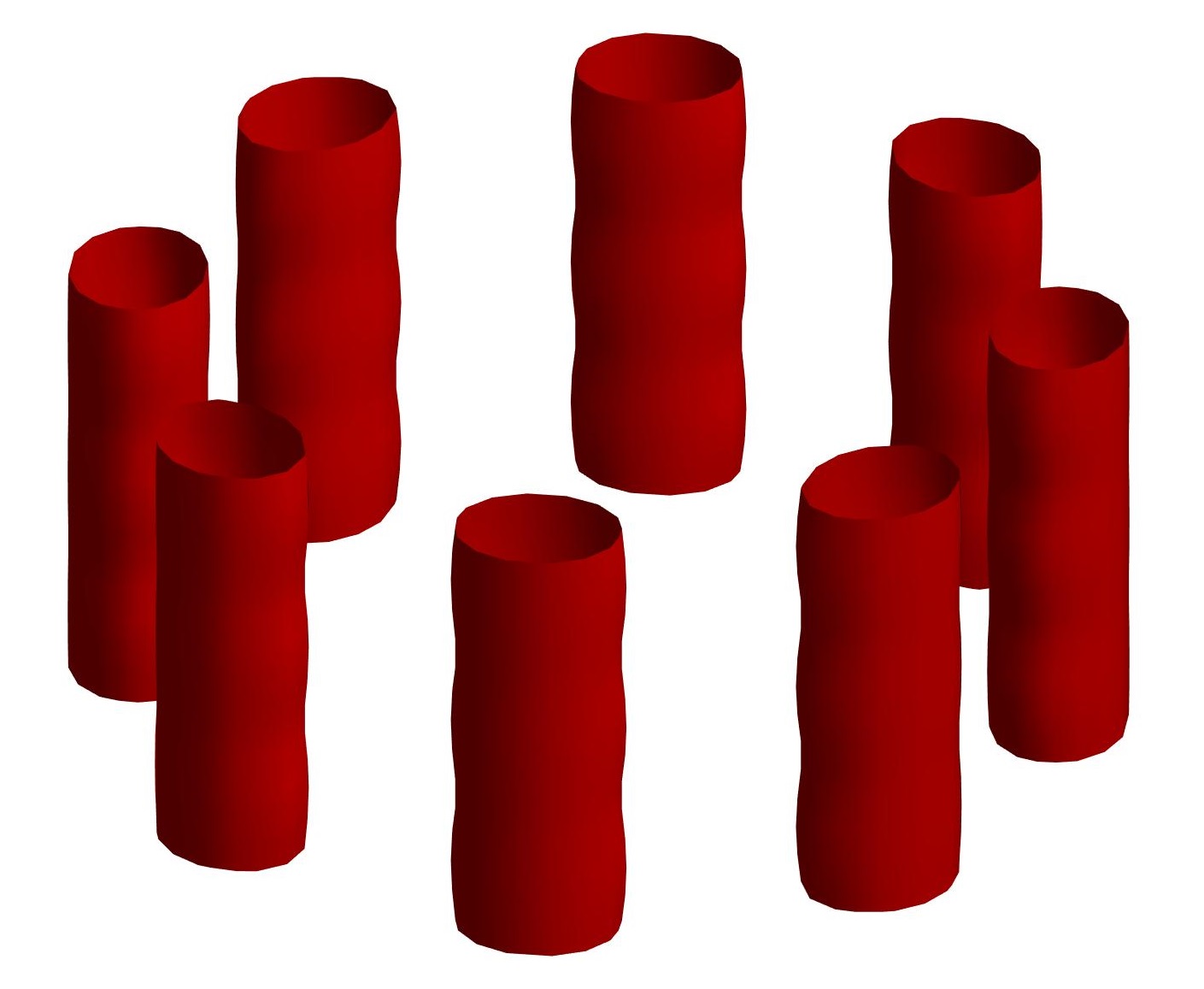}\hspace{0pt}
\includegraphics[height=110pt]{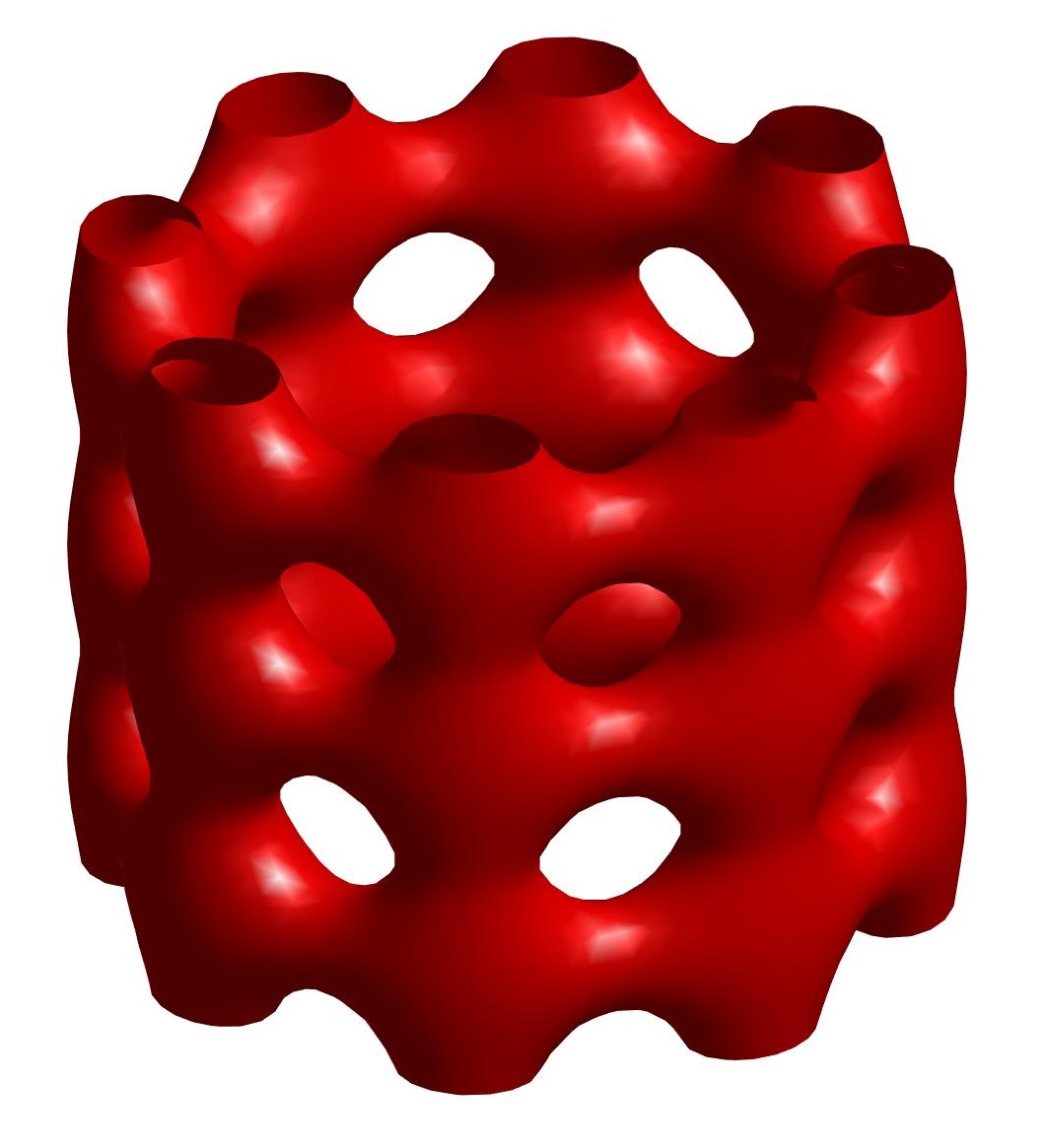}\hspace{0pt}
\includegraphics[height=110pt]{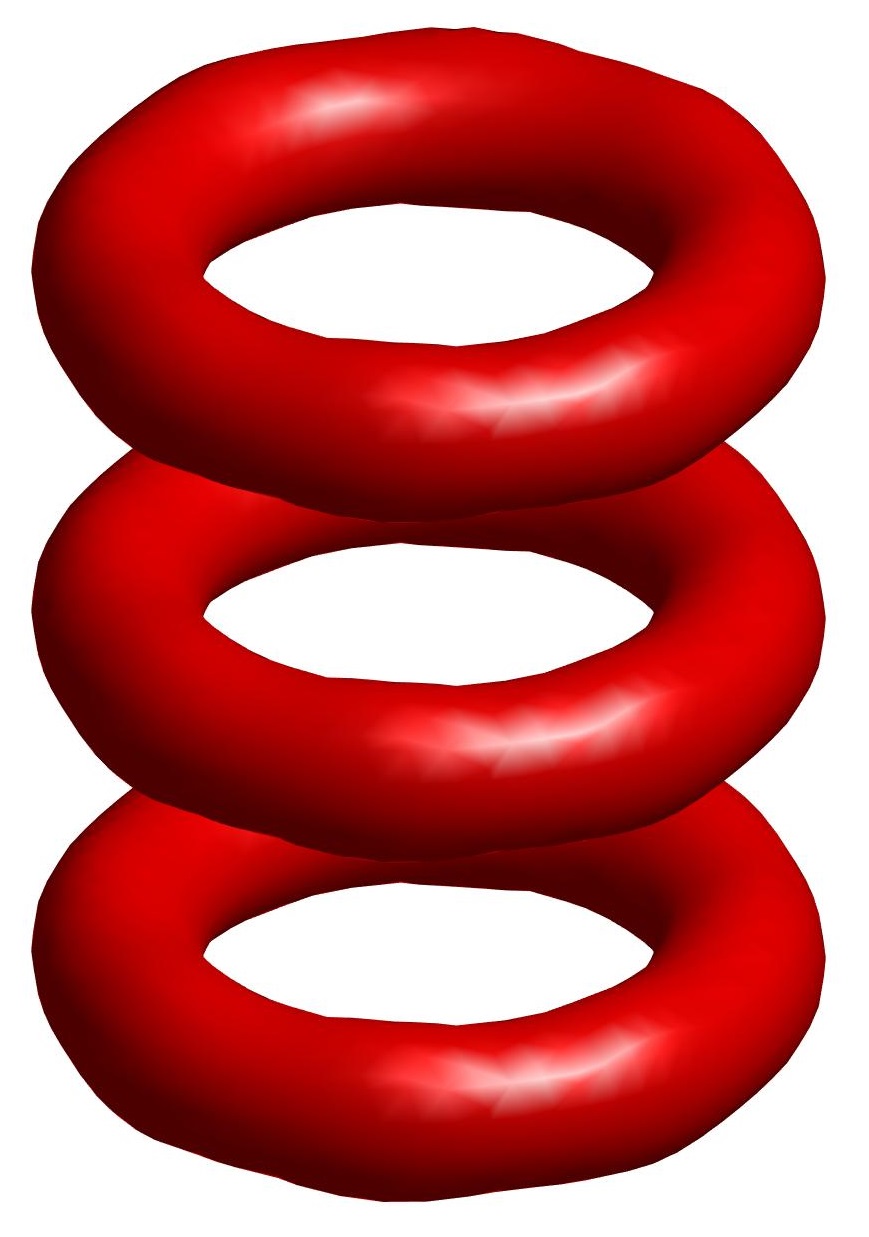}
\caption{Energy isosurfaces for $\ZZ_{8}^{(0)}$-symmetric 4-monopole chains with $c=1$ and $\beta/2\pi=0.07$ (left), $0.14$ (centre) and $0.28$ (right).  The range of the vertical axis is $3\beta$, so the image covers three ``periods''.  Images are not to scale.  The isosurface shown is where the energy density is 0.6 times its maximum value.}
\label{fig:l0}
\end{center}
\end{figure}

\begin{figure}[t]
\begin{center}
\includegraphics[height=144pt]{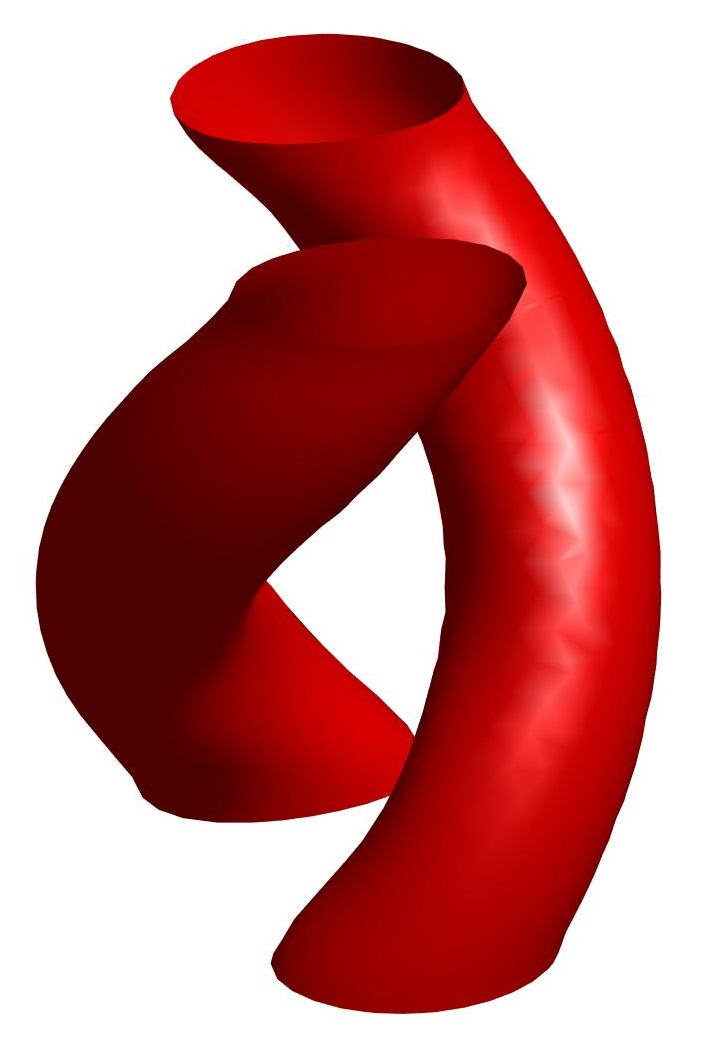}\hspace{18pt}
\includegraphics[height=144pt]{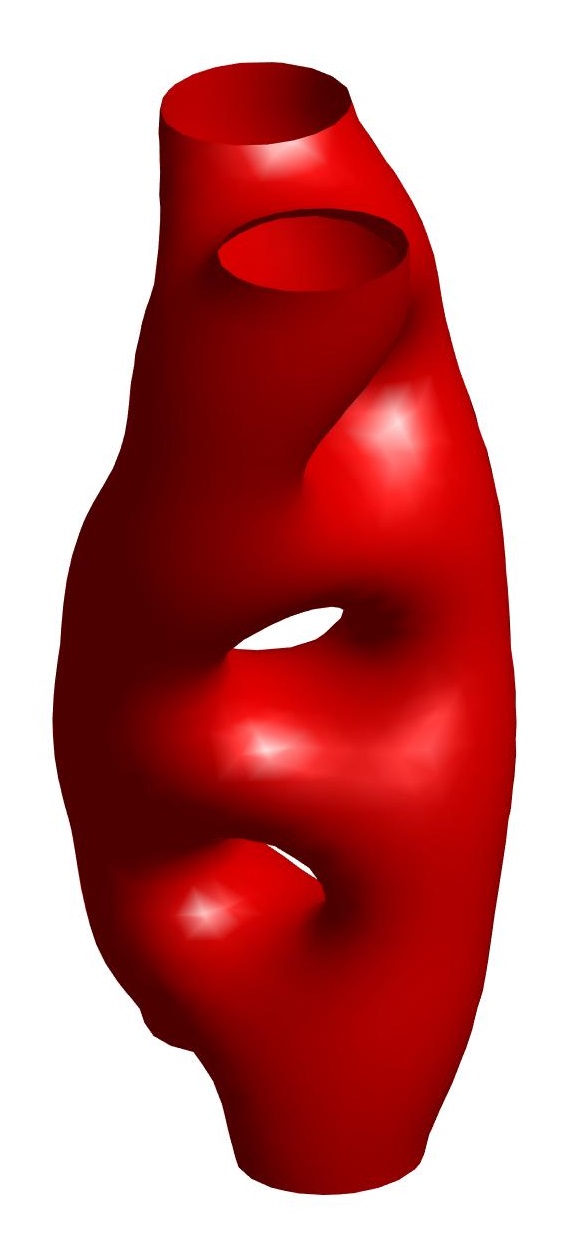}\hspace{18pt}
\includegraphics[height=144pt]{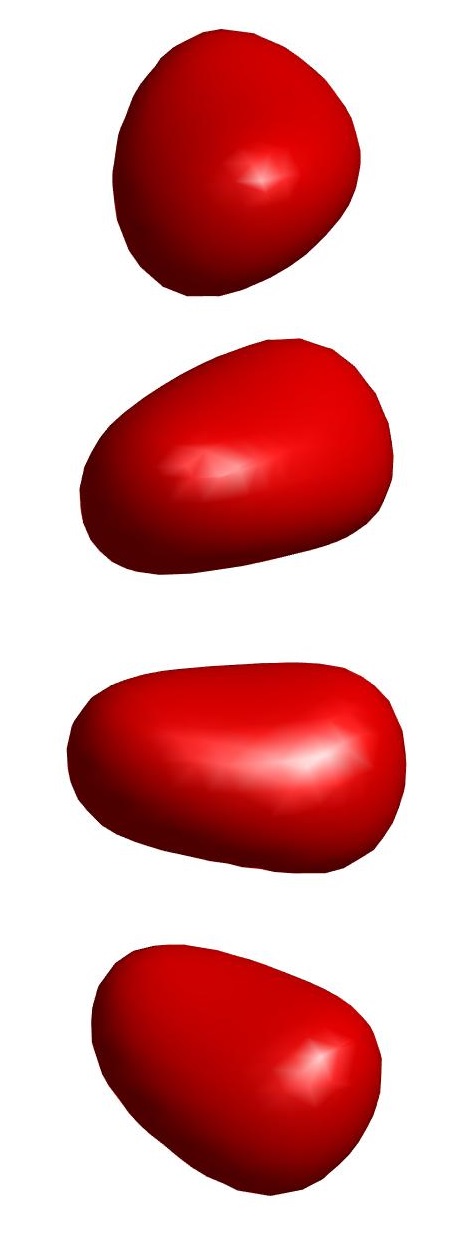}
\caption{Energy isosurfaces for $\ZZ_{8}^{(2)}$-symmetric 4-monopole chains with $c=1$ and $\beta/2\pi=1$ (left), $2$ (centre) and $3$ (right).  The range of the vertical axis is $\beta$.  Images are not to scale.  The isosurface shown is where the energy density is 0.6 times its maximum value.}
\label{fig:l1}
\end{center}
\end{figure}

\begin{figure}[t]
\begin{center}
\includegraphics[height=144pt]{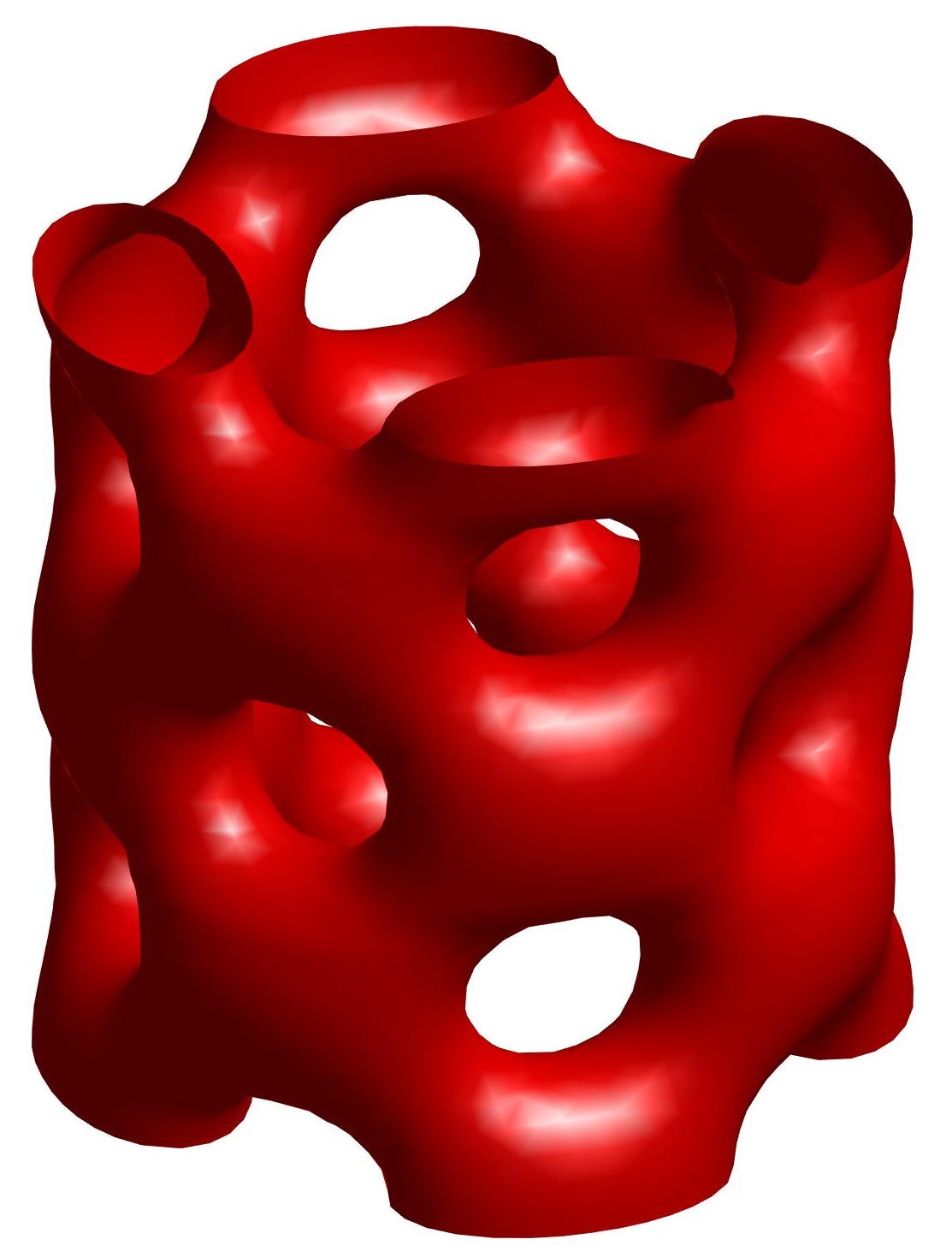}\hspace{18pt}
\includegraphics[height=144pt]{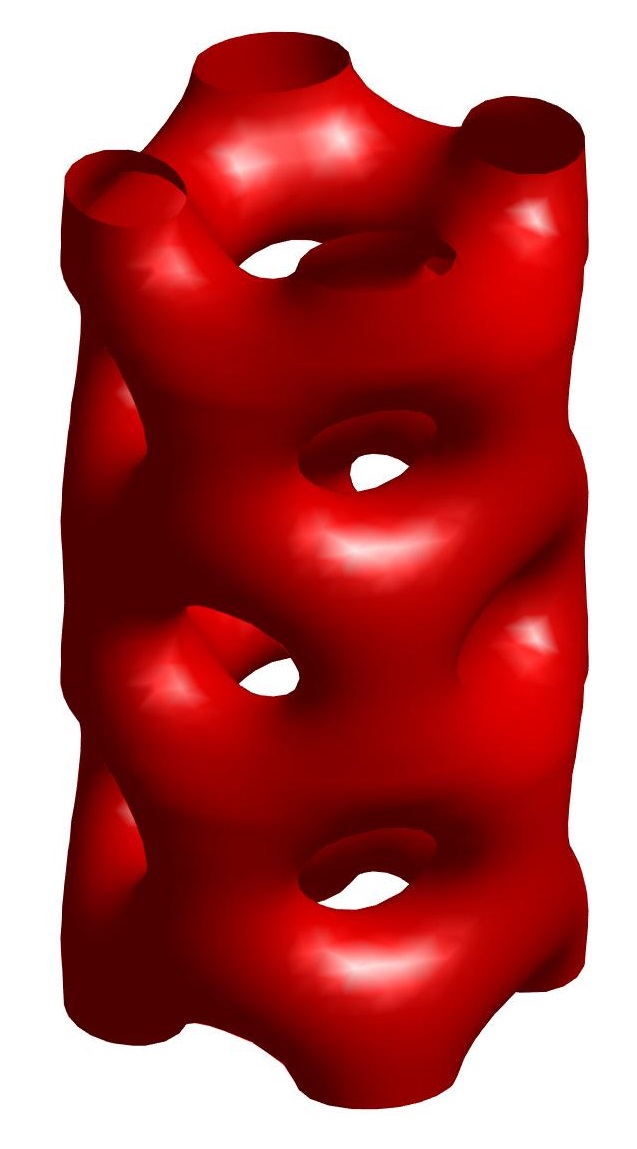}\hspace{24pt}
\includegraphics[height=144pt]{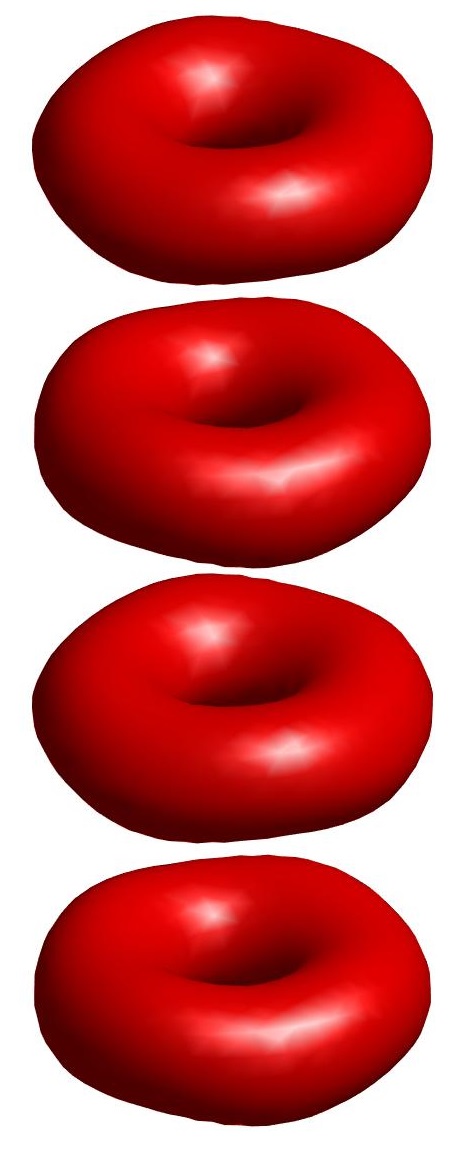}
\caption{Energy isosurfaces for $\ZZ_{8}^{(4)}$-symmetric 4-monopole chains with $c=1$ and $\beta/2\pi=0.3$ (left), $0.6$ (centre) and $1.2$ (right).  The range of the vertical axis is $2\beta$, so the image covers two ``periods''.  Images are not to scale.  The isosurface shown is where the energy density is 0.6 times its maximum value.}
\label{fig:l2}
\end{center}
\end{figure}

\subsection{Features of the monopole chains}

Energy density isosurfaces of monopole chains with $k=4$ and various values of $l$ and $\beta$ are shown in figures \ref{fig:l0}, \ref{fig:l1} and \ref{fig:l2}.  The $\ZZ_{8}^{(6)}$-symmetric chain is not shown, but this corresponds to a reflection of the $\ZZ_{8}^{(2)}$-symmetric chain in figure \ref{fig:l1}.  These are representative of images obtained for other values of $k$ (images of $\ZZ_4$-symmetric 2-monopole chains can be found in \cite{maldonado3}).  These images clearly exhibit the expected order 8 symmetry, and appear to have have additional symmetries: for example, figures \ref{fig:l0} and \ref{fig:l2} have some reflection symmetries.

For large values of $\beta$ the images resemble chains of well-separated monopoles, such that in the limit $\beta\to\infty$ one could obtain a monopole on $\mathbb{R}^3$.  The individual monopoles have charges 4, 1 and 2 in the cases of $\ZZ_{8}^{(0)}$, $\ZZ_{8}^{(2)}$ and $\ZZ_{8}^{(4)}$ symmetry.  The general pattern seems to be that a $\ZZ_{2k}^{(2l)}$-symmetric $k$-monopole chain breaks up into individual monopoles of charge $\gcd(k,l)$.  For small values of $\beta$ the monopole chains in figures \ref{fig:l0} and \ref{fig:l1} appear to separate to codimension-2 defects.

The middle pictures in figures \ref{fig:l0} and \ref{fig:l2} consist of cylindrical shells with hollow interiors, and can be considered examples of magnetic bags \cite{bolognesi}.  The magnitude of the scalar field $\hat{\phi}$ is close to zero on the interior of the shell, and seems to attain the value zero at isolated points on the central axis.  So these are ``cherry bags'', in the terminology introduced in \cite{bhs}.

Finally, we note that the middle image of figure \ref{fig:l2} is very similar to a picture of a Skyrme chain obtained in \cite{hw1}.  It remains to be seen whether any of the other monopole chains constructed here correspond to Skyrme chains.

\newpage

\appendix

\section{Proof of theorem \ref{thm:KH}, parts (i) and (ii)}

This appendix proves parts (i) and (ii) of theorem \ref{thm:KH} using results of Simpson \cite{simpson}.  We begin by considering part (ii), which concerns the behaviour of hermitian-Einstein metrics near parabolic points.

\begin{proposition}\label{prop:cyclic local phi}
Let $(E,\phi)$ be a rank $k$ cylinder Higgs bundle and let $z$ be a local holomorphic coordinate on $\CP^1$ such that $z=0$ is a parabolic point.  Then there exists a compatible holomorphic trivialisation for $E$ near 0 and local holomorphic functions $\phi_{-1},\phi_0,\ldots,\phi_{k-2}$ such that $\phi_{-1}(0)\neq0$ and
\begin{equation}\label{cyclic local phi}
\phi(z)=\sum_{i=-1}^{k-2}\phi_i(z)Z(z)^i\text{ where }Z(z) = \left(\begin{array}{ccc|c}0&\cdots&0&z\\ \hline &&&0 \\ &\mathrm{Id}_{k-1}&&\vdots \\ &&&0 \end{array}\right).
\end{equation}
\end{proposition}

\begin{proof}
We start by choosing a compatible trivialisation of $E$ near $z=0$ as in part (iii) of proposition \ref{prop:HBC}.  The eigenvalues and eigenvectors of $\phi$ are not single-valued functions of $z$, so we work on the $k$-fold covering with coordinate $u$, such that $z=cu^k$.  Then it can be shown that $\phi(cu^k)$ has an eigenvector of the form
\[ \sigma(u) = \begin{pmatrix} u^{k-1}+O(u^{k}) \\ \vdots \\ u+O(u^2) \\ 1 + O(u) \end{pmatrix}. \]
(The corresponding eigenvalue is $1/u+O(1)$.)  Let $\omega=e^{2\pi\ii/k}$; then $\sigma_i(u):=\sigma(\omega^iu)$ for $i=0,\ldots,k-1$ are eigenvectors of $\phi$ and form a local frame.

Let $\tau_i(u)$ be the frame
\[ \tau_i(u) = \begin{pmatrix} (\omega^i u)^{k-1} \\ \vdots \\ \omega^i u \\ 1 \end{pmatrix},\quad i=0,\ldots,k-1, \]
and let $g(u)$ be the invertible matrix-valued function such that $g(u)\tau_i(u)=\sigma_i(u)$.  Then 
\[ g(\omega u)\tau_i(u) = g(\omega u)\tau_{i-1}(\omega u) = \sigma_{i-1}(\omega u) = \sigma_i(u) = g(u)\tau_i(u),\]
so $g(\omega u)= g(u)$ and $g$ can be written as a function of $z=cu^k$.  Therefore $g$ defines a change of trivialisation away from the point $z=0$.  By construction, the eigenvectors of $g^{-1}\phi g$ are $\tau_i(u)$, and it follows that $g^{-1}\phi g$ can be written in the form \eqref{cyclic local phi}.

It remains to show that the trivialisation defined by $g$ is compatible with the parabolic structure at $z=0$.  It suffices to show that $g$ extends to $z=0$, and that $g(0)$ is invertible and lower-triangular.

To this end, note that $\tau_i$ can be written $\tau_i=\sum_je_jV_{ji}$, where $e_j$ are the standard basis vectors for $\CC^k$ and $V_{ji}=(\omega^iu)^{k-1-j}$ are entries of a Vandermonde matrix (with $0\leq i,j<k$).  
From definition of $\sigma_i(u)$ above, and the fact that $(\omega^ju)^{k-1-l}(V^{-1})_{ji}=V_{lj}V^{-1}_{ji}=\delta_{li}$, we obtain that
\[
ge_i = g\tau_j(V^{-1})_{ji}=\sigma_j(V^{-1})_{ji} = \begin{pmatrix}  \vdots \\ O(u^k) \\ 1 \\ O(1) \\ \vdots \end{pmatrix},
\]
with ``1'' in the $i$th position.  Thus at the point $cu^k=z=0$, $g$ is a lower triangular matrix.

Finally, we note that the holomorphic function $\phi_{-1}(z)$ in equation \eqref{cyclic local phi} cannot vanish at 0, because by definition $\phi$ has a pole at the origin.
\end{proof}

Consider the following metric, defined near a parabolic point $P$ using the trivialisation provided by proposition \ref{prop:cyclic local phi}:
\begin{equation}\label{def:h0}
h_P = \diag(|z|^{2\alpha^0_P},\ldots,|z|^{2\alpha^{k-1}_P}).
\end{equation}
This is compatible with the parabolic structure at 0.  The Chern connection of this metric is
\begin{equation}
A^{h_P}=h_P^{-1}\pa h_P = \diag(\alpha^0_P,\ldots,\alpha^{k-1}_P)\frac{\dd z}{z}.
\end{equation}
It has zero curvature, and non-trivial holonomy about the parabolic point (as the $\alpha^i_P$ are not all zero).

It will prove convenient to consider the same metric in an alternative, non-periodic gauge.  Let $s=r+\ii t$ be a local holomorphic coordinate such that $z=e^{-\beta s}$.  We apply a holomorphic gauge transformation $\phi\mapsto g^{-1}\phi g$, $h_P\mapsto g^\dagger h_P g$ with
\begin{equation}
g=\exp\left(\diag(\alpha_P^0,\ldots\alpha_P^{k-1})\beta s \right).
\end{equation}
The resulting expressions for $\phi$, $h_P$ and $A^{h_P}$ are
\begin{equation}
\label{quasi-periodic 1}
\phi=\sum_{j=-1}^{k-2}\phi_j(e^{-\beta s})e^{-\beta sj/k}\Sh^j,\quad h_P=\mathrm{Id}_k,\quad A^{h_P}=0,
\end{equation}
where
\begin{equation}
\label{def:shift}
\Sh=\left(\begin{array}{ccc|c}0&\cdots&0&1\\ \hline &&&0 \\ &\mathrm{Id}_{k-1}&&\vdots \\ &&&0 \end{array}\right).
\end{equation}
This gauge is quasi-periodic, in the sense that local sections of $E$ are represented by vector-valued functions $v(r,t)$ satisfying
\begin{equation}
\label{quasi-periodic 2}
v(r,t+2\pi/\beta) = \exp(-2\pi\ii\diag(\alpha_P^0,\ldots,\alpha_P^{k-1})) v(r,t).
\end{equation}
Note that $H^{\ast h_P}=H^{\dagger}=H^{-1}$.  It follows that $[\phi,\phi^{\ast h_P}]=0$.  Moreover, since $F^{h_P}=0$ the metric $h_P$ is hermitian-Einstein.

We will use the trivialisation just defined to study the decay properties of a second hermitian-Einstein metric.  Before doing so we state and prove a useful lemma:
\begin{lemma}\label{lemma:decay}
\begin{enumerate}
\item[(a)] Let $f:[0,\infty)\to\RR_{\geq 0}$ be a bounded non-negative real function satisfying $f''\geq m^2 f$ for some $m>0$.  Then $f'(r)\leq 0$ for all $r$ and $f(r)=O(e^{-mr})$ as $r\to\infty$.
\item[(b)] Let $F:[0,\infty)\times S^1\to\RR_{\geq 0}$ be a bounded non-negative real function satisfying $\triangle F \geq 0$, such that $\int_{S^1}F(r,t)\dd t=O(e^{-mr})$ as $r\to\infty$ for some $m>0$.  Then $\sup_t F(r,t)=O(e^{-mr})$ as $r\to\infty$.
\item[(c)] Let $F:[0,\infty)\times S^1\to\RR_{\geq 0}$ be a bounded non-negative real function satisfying $\triangle F \geq m^2 F$ for some $m>0$.  Then $\sup_t F(r,t)=O(e^{-mr})$ as $r\to\infty$.
\end{enumerate}
\end{lemma}
\begin{proof}
(a) Consider the function $g(r)=f'(r)+mf(r)$.  This satisfies $g'\geq mg$.  Suppose that $g(a)>0$ for some $a\in[0,\infty)$.  Then $g(r)\geq g(a)e^{m(r-a)}$ for all $r\geq a$.  This implies that $f$ grows exponentially and contradicts the boundedness of $f$.  So $f'(r)+mf(r)\leq 0$ for all $r$.  It follows that $f'\leq 0$, and that $f=O(e^{-mr})$.

(b) Choose a positive constant $C$ such that $\int_{S^1}F(r,t)\dd t\leq Ce^{-mr}$.  Since $\triangle F\geq0$, the value of $F$ at any point $r_0,t_0$ is bounded from above by its average over a ball of radius $\epsilon>0$.  Therefore
\begin{multline}
F(r_0,t_0)\leq \frac{1}{\pi\epsilon^2}\int_{B_\epsilon(r,t)}F(r,t)\dd r\,\dd t \leq \frac{1}{\pi\epsilon^2}\int_{|r-r_0|\leq \epsilon} F(r,t)\dd r\,\dd t \\
\leq \frac{1}{\pi\epsilon^2}\int_{r_0-\epsilon}^{r_0+\epsilon} Ce^{-mr}\dd r=O(e^{-mr}).
\end{multline}

(c) Consider the function $f(r)=\int_{S^1}F(r,t)\dd t$.  This satisfies $f''(r)\geq m^2f(r)$, so by part (a) $f(r)=O(e^{-mr})$ as $r\to\infty$, and by part (b) $\sup_t F(r,t)=O(e^{-mr})$ as $r\to\infty$.
\end{proof}

Now we consider the decay properties of a hermitian-Einstein metric.  In order to state the next result we need to recall the definition of the Donaldson-Simpson functional.  Let $(E,\Phi)$ be a Higgs bundle over a (possibly non-compact) Riemann surface $\RS$.  Let $h_0$ and $h_1$ be two hermitian metrics such that $h_1(\cdot,\cdot)=h_0(\cdot,e^{\psi}\cdot)$ for some section $\psi$ of $\End(E)$ which his hermitian with respect to $h_0$.  Choose a local frame for $E$ consisting of eigenvectors of $\psi$ which is orthonormal with respect to $h_0$, and let $\lambda_i$ denote the associated eigenvalues.  Given any section $X$ of $\mathrm{End}(E)$, we denote the matrix components of $X$ with respect to this basis by $X_{ij}$.  Let $\rho$ be the real analytic function
\begin{equation}\label{def:f}
\rho(x)=\begin{cases} \frac{e^x-x-1}{x^2} & x\neq0 \\ \frac12 & x=0 \end{cases}
\end{equation}
The Donaldson-Simpson density is defined by
\begin{align}\label{DS}
\mathcal{DS}(h_0,h_1) &= \ii\Tr(\psi\mathcal{F}^{h_0})-\ii\sum_{i,j}\rho(\lambda_j-\lambda_i)D''\psi_{ji}\wedge D'_{h_0}\psi_{ij} \\
&= \ii\Tr(\psi (F^{h_0}+\Phi\wedge\Phi^{\ast h_0}))\nonumber \\
& + \sum_{i,j}\rho(\lambda_j-\lambda_i)\left\{\bar{\pa}\psi_{ji}\wedge\ast(\bar{\pa}\psi)^{\ast h_0}_{ij} + [\Phi,\psi]_{ji}\wedge\ast[\Phi,\psi]^{\ast h_0}_{ij}\right\}.\nonumber
\end{align}
Here $\ast$ denotes the Hodge star with respect to some metric compatible with the complex structure; in particular, $\ast$ acts as multiplication by $-\ii$ (resp.\ $\ii$) on $\Lambda^{1,0}$ (resp.\ $\Lambda^{0,1}$).  Note that the 2-form $\mathcal{DS}(h_0,h_1)$ does not not depend on the choice of orthonormal basis, and is thus defined globally, even though the orthonormal frame may exist only locally.  The Donaldson-Simpson functional is defined by
\begin{equation}
DS(h_0,h_1) = \int_\RS \mathcal{DS}(h_0,h_1)
\end{equation}
Simpson proved the existence of hermitian-Einstein metrics on stable Higgs bundles by studying the gradient flow for this functional.  Note that although our definition of the Donaldson-Simpson functional involves a choice of metric, it is independent of this choice since the action of $\ast$ on 1-forms is conformally invariant.

We are now ready to state our first result on the decay of hermitian-Einstein metrics:
\begin{proposition}\label{prop:decay 1}
.Let $E$ be a cylinder Higgs bundle and let $z=e^{-\beta s}$ be a local holomorphic coordinate such that $z=0$ is a parabolic point $P$, and write $s=r+\ii t$ for real coordinates $r,t$.  Let $h_P$ be a hermitian metric defined as in \eqref{def:h0} and let $\psi$ be a bounded traceless self-adjoint section of $\End(E)$ such that $h(\cdot,e^\psi\cdot)$ is a hermitian-Einstein metric.  Then
\begin{enumerate}
\item[(a)] $\psi$ decays uniformly and exponentially in $r$ as $r\to\infty$;
\item[(b)] $\int_{\{r\geq R\}}\mathcal{DS}(h_P,h)<\infty$ for sufficiently large $R$; and
\item[(c)] $|D''\psi|^2$ and $|D''D'_{h_P}\psi|$ are integrable functions, where the norms and integrals are defined using the metric $h_P$ and the cylindrical metric $\dd r^2+\dd t^2$.
\end{enumerate}
\end{proposition}
\begin{proof}
Let $\lambda_i$ denote the eigenvalues of $\psi$ and let us choose a local frame consisting of eigenvectors of $\psi$ which is orthonormal with respect to $h_P$.  Again, we write $X_{ij}$ for the matrix elements of an endomorphism $X$ with respect to this frame.

The following identity plays a key role in this proof:
\begin{equation}\label{identity}
(e^{-\psi} D'_{h_P}e^\psi)_{ij} = \frac{e^{\lambda_j-\lambda_i}-1}{\lambda_j-\lambda_i}D'_{h_P}\psi_{ij}.
\end{equation}
Here the right hand side is understood to equal $D'_{h_P}\psi_{ij}$ if $\lambda_i=\lambda_j$.  We now present a short proof of this identity.

First, the identity $[\psi,e^\psi]=0$ implies that $[D'_{h_P}e^\psi,\psi]=[D'_{h_P}\psi,e^{\psi}]$ and hence that
\begin{equation}
[e^{-\psi} D'_{h_P}e^\psi,\psi]=e^{-\psi}[D'_{h_P}e^\psi,\psi] = e^{-\psi}[D'_{h_P}\psi,e^{\psi}].
\end{equation}
This implies that $(\lambda_j-\lambda_i)(e^{-\psi}D'_{h_P}e^\psi)_{ij}=(e^{\lambda_j-\lambda_i}-1)D'_{h_P}\psi_{ij}$, and if $\lambda_j\neq\lambda_i$ the result follows.  Second, consider the case that $\lambda_i=\lambda_j$ for some $i,j$ at a given point $z$.  Without loss of generality we may assume that $\lambda_i(z)=\lambda_j(z)=0$, since neither side of equation \eqref{identity} is changed by adding a multiple of the identity matrix to $\psi$.  We have that
\begin{equation}
e^\psi = \mathrm{Id}_k + \psi + \psi \rho(\psi) \psi,
\end{equation}
where $\rho$ is the analytic function defined in equation \eqref{def:f}.  Then
\begin{equation}
D'_{h_P}e^\psi = D'_{h_P}\psi + D'_{h_P}(\psi \rho(\psi)) \psi + \psi \rho(\psi) D'_{h_P}\psi.
\end{equation}
We are interested in the $i,j$-component of this matrix equation.  Since $\psi$ is diagonal and $\psi_{ii}=\psi_{jj}=0$ at the point $z$, the $i,j$-component is $(e^{-\psi}D'_{h_P}e^\psi)_{ij}(z) = D'_{h_P}\psi_{ij}(z)$, which was to be proved.

Let $\mathcal{N}$ be the non-negative real function $\mathcal{N}=\Tr(\psi^2)$ and let $n(r)=\int_{S^1}\mathcal{N}(r,t)\dd t$.  We aim to prove the following two identities for some positive constant $C$:
\begin{align}
\label{identity 2}\triangle\mathcal{N}&\geq4\ast \mathcal{DS}(h_P,h) \\
\label{identity 3}\int_{S^1}\mathcal{DS}(h_P,h)(r,t)\dd t &\geq C n(r).
\end{align}
Here $\triangle=\pa_r^2+\pa_t^2$ is the Laplacian with respect to the cylindrical metric $\dd r^2+\dd t^2$ and $\ast$ denotes Hodge star for the same metric.  It follows from these and from lemma \ref{lemma:decay} that $n(r)$ decays exponentially as $r\to\infty$ and that $n'(r)\leq0$ for all $r$.  Since $\triangle\mathcal{N}\geq0$ the lemma shows moreover that $\sup_{t}\mathcal{N}(r,t)$ decays exponentially with $r$, as claimed in (a).  Moreover, claim (b) will also follow because the fact that $n'(r)\leq0$ will imply that
\begin{equation}
4\int_{r\geq R}\mathcal{DS} \leq \int_{r\geq R}\triangle\mathcal{N}\dd r\dd t = \lim_{r\to\infty}n'(r)-n'(R)\leq -n'(R)<\infty.
\end{equation}

First we prove identity \eqref{identity 2}.  Differentiating $\mathcal{N}$ yields
\begin{equation*}
\pa\mathcal{N} = \Tr\pa^{h_P}(\psi^2)=\Tr D'_{h_P}(\psi^2)=2\Tr(\psi D'_{h_P}\psi)=2\sum_i\lambda_i(D'_{h_P}\psi)_{ii}
\end{equation*}
since $[\Phi^{\ast h_P},\psi^2]$ is traceless.  Then, by equation \eqref{identity},
\begin{equation*}
\pa\mathcal{N}=2\sum_i\lambda_i(e^{-\psi}D'_{h_P}e^\psi)_{ii}=2\Tr(\psi e^{-\psi}D'_{h_P}e^\psi).
\end{equation*}
The hermitian-Einstein equation for $h$ reads
\begin{equation*}
0=\mathcal{F}^{h}=\mathcal{F}^{h_P}+D''(e^{-\psi}D'_{h_P}e^\psi).
\end{equation*}
Therefore differentiating $\mathcal{N}$ again and employing the identity \eqref{identity} yields
\begin{align*}
\bar{\pa}\pa\mathcal{N}&=2\Tr(D''(\psi e^{-\psi}D'_{h_P}e^\psi)) \\
&= 2\Tr(D''\psi\wedge e^{-\psi}D'_{h_P}e^\psi - \psi\mathcal{F}^h) \\
&= 2\sum_{i,j}\frac{e^{\lambda_j-\lambda_i}}{\lambda_j-\lambda_i}D''\psi_{ji}\wedge D'_{h_P}\psi_{ij}-2\Tr(\psi\mathcal{F}^{h_P}).
\end{align*}
The identity \eqref{identity 2} for $\triangle\mathcal{N}=-2\ii\ast\bar{\pa}\pa\mathcal{N}$ then follows from the definition \eqref{DS} of the Donaldson-Simpson density and the fact that the analytic function
\begin{equation*}
\frac{e^x-1}{x}-\frac{e^x-x-1}{x^2}
\end{equation*}
is non-negative for all $x\in\RR$.


Now we prove identity \eqref{identity 3}.  Note that, by construction, $F^{h_P}=0$ and $[\Phi,\Phi^{\ast h_P}]=0$.  Since $\psi$ is bounded we know that there is a positive constant $C_1$ such that
\begin{equation}
\label{bound on L2 norm}
\ast \mathcal{DS}(h,h_P) \geq C_1 \ast \Tr \left(\bar{\pa}\psi\wedge\ast(\bar{\pa}\psi)^{\ast h_0} + [\Phi,\psi]\wedge\ast[\Phi,\psi]^{\ast h_0}\right).
\end{equation}
We work in the gauge of equations \eqref{quasi-periodic 1} and \eqref{quasi-periodic 2}.  In this gauge $\psi$ need not be diagonal.  We write $\psi=\psi^D+\psi^\perp$, where $\psi^D$ is diagonal and $\psi^{\perp}$ has zeros on its diagonal.  As a result of equations \eqref{quasi-periodic 2} and \eqref{PHB weights definition}, the entries of $\psi$ satisfy
\begin{equation}\label{eq:psi quasi-periodic}
\psi_{ij}(r,t+2\pi/\beta) = e^{2\pi\ii(\alpha_P^j-\alpha_P^i)}\psi_{ij}(r,t) = e^{2\pi\ii(j-i)/k}\psi_{ij}(r,t).
\end{equation}
Thus the entries of $\psi^\perp$ are quasi-periodic.  By considering the Fourier series or otherwise we deduce that
\begin{equation}\label{inequality 1}
\int_{S^1}\ast\Tr(\bar{\pa}\psi\wedge\ast(\bar{\pa}\psi)^{\ast h_P})\dd t \geq \frac{1}{4}\int_{S^1}\Tr(\pa_t\psi\pa_t\psi^{\ast h_P})\dd t \geq C_2\int_{S^1}\Tr((\psi^\perp)^2).
\end{equation}
for a positive constant $C_2$ that depends on $k$.

Now consider the term $\ast\Tr([\Phi,\psi]\wedge\ast[\Phi,\psi]^{\ast h_P})=\frac{1}{2}\Tr(\psi[\phi,[\phi^{\ast h_P},\psi]])$.  Recall from equation \eqref{quasi-periodic 1} that the leading term in $\phi$ as $r\to\infty$ is the matrix $e^{\beta s/k}\Sh^{-1}$, where $\Sh$ is given in equation \eqref{def:shift}.  It is straightforward to check that the self-adjoint operator $[\Sh^{-1},[(\Sh^{-1})^{\ast h_P},\cdot\,]]=[\Sh^{-1},[\Sh,\cdot\,]]$ acting on traceless diagonal matrices has positive eigenvalues.  Therefore there exists a constant $C_3$ such that for sufficiently large $r$
\begin{equation}
\label{inequality 2}
\ast \Tr([\Phi,\psi]\wedge\ast[\Phi,\psi]^{\ast h_P})\geq C_3 \Tr((\psi^{\perp})^2).
\end{equation}
Identity \eqref{identity 3} follows from inequalities \eqref{inequality 1} and \eqref{inequality 2}.

Having established parts (a) and (b) of the proposition, we now prove part (c).  By equation \eqref{bound on L2 norm} the integral of $|D''\psi|^2$ is bounded above by a multiple of the Donaldson-Simpson functional, which is finite by part (b).  We establish integrability of $|D''D'\psi|$ using the following identity, whose proof is similar to that of \eqref{identity}:
\begin{multline}
\frac{e^{\lambda_j-\lambda_i}-1}{\lambda_j-\lambda_i}(D''D'_{h_P}\psi)_{ij} = \mathcal{F}^{h_P}_{ij} \\
+ \frac{e^{\lambda_j}}{\lambda_i-\lambda_j}\left(\frac{e^{-\lambda_i}-e^{-\lambda_k}}{\lambda_i-\lambda_k}-\frac{e^{-\lambda_j}-e^{-\lambda_k}}{\lambda_j-\lambda_k}\right)D''\psi_{ik}\wedge D'_{h_P}\psi_{kj} \\
+ \frac{e^{-\lambda_i}}{\lambda_i-\lambda_j}\left(\frac{e^{\lambda_i}-e^{\lambda_k}}{\lambda_i-\lambda_k}-\frac{e^{\lambda_j}-e^{\lambda_k}}{\lambda_j-\lambda_k}\right)D'_{h_P}\psi_{ik}\wedge D''\psi_{kj}.
\end{multline}
As has already been observed, $\mathcal{F}^{h_P}=0$.  The coefficients of the remaining three terms are analytic functions which remain bounded as $\lambda_i-\lambda_j$, $\lambda_i-\lambda_k$ or $\lambda_k-\lambda_j$ approach 0.  Moreover, the coefficient $(e^{\lambda_j-\lambda_i}-1)/(\lambda_j-\lambda_i)$ of $(D''D'_{h_P}\psi)_{ij}$ never vanishes.  Since $\psi$ is bounded, it follows that $|D''D'_{h_P}\psi|$ is bounded from above by a constant multiple of $|D''\psi|^2$ and thus is integrable.
\end{proof}

To prove our next result on the decay of hermitian metrics, we need the following lemma:
\begin{lemma}\label{lemma:Vi}
Let $\Sh$ be the $k\times k$ matrix defined in equation \eqref{def:shift}.  Consider the following functions on the space of $k\times k$ hermitian matrices:
\begin{align}
V_1(\psi) &= \Tr(e^{-\psi}\Sh^{-1}e^\psi-\Sh^{-1})(\Sh-e^{-\psi}\Sh e^{\psi}) \\
V_2(\psi) &= \Tr([\Sh^{-1},e^{-\psi}\Sh e^\psi][\Sh ^{-1},e^{-\psi}\Sh e^\psi]) \\
V_3(\psi) &= \Tr([\Sh^{-1},e^{-\psi}\Sh e^\psi][\Sh ,e^{-\psi}\Sh^{-1}e^\psi]).
\end{align}
Then there exist constants $\epsilon>0$ and $C>1$ such that
\begin{equation}
|\psi|<\epsilon \implies 0\leq V_i(\psi) < CV_j(\psi) \quad\forall i,j\in\{1,2,3\}.
\end{equation}
\end{lemma}
\begin{proof}
Let $U$ denote the set of all $k\times k$ hermitian matrices and let $\Delta$ denote the set of hermitian matrices that commute with $\Sh$.  One can show that $\Delta$ is equal to the intersection of the span of the linearly-independent matrices $\mathrm{Id}_k, \Sh, \Sh^2,\ldots,\Sh^{k-1}$ with $U$.
We will show that:
\begin{enumerate}
\item $V_i(\psi)=0$ and $\dd V_i(\psi)=0$ for $\psi\in \Delta$.
\item The quadratic forms $Q_i$ on $U/\Delta$ defined by the hessians of $V_i$ at $0$ are positive definite.
\end{enumerate}
The result then follows by considering the Taylor expansions of $V_i$ about points $\psi\in\Delta$.

Item 1 follows almost immediately from the observation that the matrices
\begin{equation}
e^{-\psi}\Sh^{-1}e^\psi-\Sh^{-1},\quad \Sh-e^{-\psi}\Sh e^\psi,\quad [\Sh^{-1},e^{-\psi}\Sh e^\psi],\quad [\Sh,e^{-\psi}\Sh^{-1}e^\psi]
\end{equation}
all vanish when $\psi\in\Delta$.

For item 2, we introduce the operator $T(X)=[\Sh,[\Sh^{-1},X]]=[\Sh^{-1},[\Sh,X]]$ acting on hermitian matrices $X$.  A short calculation shows that the hessians of $V_1,V_2,V_3$ are
\begin{align}
\frac{\dd^2}{\dd t^2} V_1(tX)\big|_{t=0}&=\Tr(XT(X)),\\
\frac{\dd^2}{\dd t^2} V_2(tX)\big|_{t=0}=\frac{\dd^2}{\dd t^2} V_3(tX)\big|_{t=0}&=\Tr(XT^2(X)).
\end{align}
Since $\Sh$ is unitary it can be diagonalised, and the eigenvalues of $T$ are of the form $|\lambda_i-\lambda_j|^2$ for eigenvalues $\lambda_i,\lambda_j$ of $\Sh$.  It is straightforward to check that the eigenvalues of $\Sh$ are distinct (in fact they are the roots of unity), so precisely $k$ of the eigenvalues of $T$ are zero.  The corresponding eigenspace is $\Delta$, and $T$ defines a positive definite operator on $U/\Delta$.  It follows that the quadratic forms $Q_i$ on $H/\Delta$ are positive definite.
\end{proof}

Now we state our second result about the decay of hermitian-Einstein metrics, from which part two of theorem \ref{thm:KH} follows:
\begin{proposition}\label{prop:decay 2}
Let $E$ be a cylinder Higgs bundle and let $z=e^{-\beta s}$ be a local holomorphic coordinate such that $z=0$ is a parabolic point, and write $s=r+\ii t$ for real coordinates $r,t$.  Let $h_P$ be a hermitian metric defined as in \eqref{def:h0} and let $\psi$ be a bounded traceless self-adjoint section of $\End(E)$ such that $h=h_P(\cdot,e^\psi\cdot)$ is a hermitian-Einstein metric.  Then $\sup_t|F^{h}|(r,t)$ decays faster than any exponential function as $r\to\infty$.
\end{proposition}
\begin{proof}
We work in the gauge of equations \eqref{quasi-periodic 1} and \eqref{quasi-periodic 2}.  Consider the functions
\begin{align}
\mathcal{V}_1(r,t) &= \Tr(\Sh^{\ast h}-\Sh^{-1})(\Sh^{\ast h}-\Sh^{-1})^{\ast h}\nonumber \\
&= \Tr(\Sh\Sh^{\ast h}+\Sh^{-1}(\Sh^{-1})^{\ast h})-2k \\
\mathcal{V}_2(r,t) &= \Tr([\Sh^{-1},(\Sh^{-1})^{\ast h}]^2) \\
\mathcal{V}_3(r,t) &= \Tr([\Sh^{-1},(\Sh^{-1})^{\ast h}][\Sh,\Sh^{\ast h}]).
\end{align}
These are just the compositions of the functions $V_1,V_2,V_3$ studied in lemma \ref{lemma:Vi} with $\psi$.  They are periodic, in the sense that $\mathcal{V}_i(r,t+2\pi/\beta)=\mathcal{V}_i(r,t)$.  This is most easily shown by rewriting these functions in terms of $Z=e^{-\beta s/k}\Sh$, which represents a well-defined section of $\End(E)$ over the cylinder.

We aim to estimate $\triangle \mathcal{V}_1$.  First, since $\bar{\pa}\Sh=0$,
\begin{equation}
\pa \mathcal{V}_1 = \Tr(\Sh^{\ast h}\pa^{h} \Sh + (\Sh^{-1})^{\ast h}(\pa^{h}\Sh^{-1})).
\end{equation}
Then, using again the holomorphicity of $\Sh$ and the hermitian-Einstein equation \eqref{eq:HE}, we obtain
\begin{align}
\ast\triangle \mathcal{V}_1 &= -2\ii \bar{\pa}\pa \mathcal{V}_1 \\
&= -2\ii\Tr((\pa^{h}\Sh)^{\ast h}\wedge\pa^{h} \Sh + (\pa^{h}\Sh^{-1})^{\ast h}\wedge(\pa^{h}\Sh^{-1})) \nonumber \\
&\quad -2\ii\Tr(\Sh^{\ast h}[F^{h}, \Sh] + (\Sh^{-1})^{\ast h}[F^{h},\Sh^{-1}]) \\
&= 2\Tr((\pa^{h}\Sh)^{\ast h}\wedge\ast\pa^{h} \Sh + (\pa^{h}\Sh^{-1})^{\ast h}\wedge\ast(\pa^{h}\Sh^{-1})) \nonumber \\
&\quad \ast \Tr(\Sh^{\ast h}[[\phi,\phi^{\ast h}], \Sh] + (\Sh^{-1})^{\ast h}[[\phi,\phi^{\ast h}],\Sh^{-1}]).
\end{align}
The leading term in the expansion of $\phi$ given in equation \eqref{quasi-periodic 1} is a positive multiple of $e^{\beta s/k}\Sh^{-1}$, so there exists a positive constant $C_1$ such that, for sufficiently large $r$,
\begin{align}
\triangle V_1 &\geq C_1 e^{2\beta r/k}\Tr\big(\Sh^{\ast h}[[\Sh^{-1},(\Sh^{-1})^{\ast h}], \Sh] \nonumber\\
&\qquad\qquad+ (\Sh^{-1})^{\ast h}[[\Sh^{-1},(\Sh^{-1})^{\ast h}],\Sh^{-1}]\big) \\
&= C_1 e^{2\beta r/k}\Tr([\Sh^{-1},(\Sh^{-1})^{\ast h}][\Sh,\Sh^{\ast h}] + [\Sh^{-1},(\Sh^{-1})^{\ast h}]^2 ) \\
&= C_1 e^{2\beta r/k}(\mathcal{V}_2+\mathcal{V}_3).
\end{align}
Let $\epsilon$ be the constant given in proposition \ref{lemma:Vi}.  By proposition \ref{prop:decay 1} $\psi$ satisfies the estimate $|\psi(r,t)|<\epsilon$ for sufficiently large $r$.  Therefore by lemma \ref{lemma:Vi} there exists a constant $C_2$ such that, for sufficiently large $r$,
\begin{equation}
\triangle \mathcal{V}_1 \geq C_2 e^{2\beta r/k} \mathcal{V}_1.
\end{equation}
By choosing $r$ sufficiently large the coefficient of $\mathcal{V}_1$ on the right hand side can be made arbitrarily large.  Therefore by part (c) of lemma \ref{lemma:decay} $\sup_t\mathcal{V}_1(r,t)$ decays faster than any exponential function of $r$.

Now we consider the curvature $F^{h}$ of the Chern connection of $h$.  By the hermitian-Einstein equation,
\begin{equation}
|F^{h}|^2_{h} = |[\Phi,\Phi^{\ast h}]|^2_h = \frac{1}{4}\Tr([\phi,\phi^{\ast h}]^2).
\end{equation}
By equation \eqref{quasi-periodic 1} the leading contribution to $\phi$ at large $r$ is $e^{\beta s/k}\Sh^{-1}$.  Therefore there exists a constant $C_3$ such that, for sufficiently large $r$,
\begin{equation}
|F^{h}|^2_{h} \leq C_3 e^{4\beta r/k}\Tr([\Sh^{-1},(\Sh^{-1})^{\ast h}]^2) = C_3e^{4\beta r/k}\mathcal{V}_2.
\end{equation}
By proposition \ref{prop:decay 1} and lemma \ref{lemma:Vi} there exists a positive constant $C_4$ such that, for sufficiently large $r$,
\begin{equation}
|F^{h}|^2_{h} \leq C_4 e^{4\beta r/k}\mathcal{V}_1.
\end{equation}
We have already shown that $\sup_t\mathcal{V}_1(r,t)$ decays faster than any exponential function, and it follows that $\sup_t|F^{h}|^2_{h}$ decays faster than any exponential function.
\end{proof}

Now we prove part (i) of theorem \ref{thm:KH}.  For the existence part we appeal to a theorem of Simpson \cite{simpson}.  Simpson proved existence of a hermitian-Einstein metric by applying a heat flow to the Donaldson-Simpson functional.  To use this theorem we need to show that the Higgs bundle is stable and supply a suitable initial hermitian metric.

First we show stability of $(E,\phi)$.  This entails showing that every $\phi$-invariant sub-bundle of $E$ satisfies a slope-stability condition.  We establish this trivially by showing that there are no non-trivial $\phi$-invariant sub-bundles.

To show that $E|_{\CC^\ast}$ has no $\phi$-invariant subbundles, consider the curve in $\CC^\ast\times\CC$ defined by the characteristic polynomial of $\phi$:
\[ \det(\zeta\mathrm{Id}_E-\phi(w))=0. \]
(This is the spectral curve, and will be discussed in more detail in the next section.)  The map $(w,\zeta)\mapsto w$ gives a $k$-sheeted branched covering of this curve over $\CC^\ast$, and the sheets of this covering correspond to the eigenvalues of $\phi$.  Near a parabolic point $P$ one obtains from part (\ref{HBC3}) of proposition \ref{prop:HBC} that the curve has equation
\[ 0 = \zeta^k + f_{k-1}(z)\zeta^{k-1}+\ldots+f_1(z)\zeta+f_0(z)-c/z. \]
It follows that for sufficiently small $z\neq0$ the eigenvalues of $\phi$ are distinct.  Moreover, as $z$ circles once around the point $z=0$ the eigenvalues are cyclically permuted.

Suppose that $F\subseteq E$ is a $\phi$-invariant sub-bundle.  The eigenvalues of the restriction of $\phi$ to $F$ at any point $z$ form a subset of the set of eigenvalues of $\phi$.  Since these eigenvalues depend continuously on $z$ as it circles the point $z=0$, this subset must be invariant under cyclic permutations, hence is either the whole set or the empty set.  Thus $F$ has rank either $k$ or 0, so is not a non-trivial sub-bundle.

Now we construct a suitable initial hermitian metric.  We may assume without loss of generality that $E$ has degree 0, since the degree can be changed by twisting $E$.  Since the parabolic degree of $E$ is zero, this means that the parabolic weights satisfy
\begin{equation}\label{eq:sum weights}
\sum_{P\in\{0,\infty\}}\sum_{i=0}^{k-1}\alpha_P^i = 0
\end{equation}
Let us choose a non-vanishing holomorphic section $e$ of the line bundle $\Lambda^k(E)$; doing so trivialises $\Lambda^k(E)$, and provides an identification of hermitian metrics $h$ on $\Lambda^k(E)$ with positive real functions $h(e,e)$.

Recall that the cylinder Higgs bundle has two parabolic points $P=0,\infty$, with local coordinates $z_0=w$ and $z_\infty=w^{-1}$.  Near each such point choose a compatible frame $e_0^P,\ldots,e_{k-1}^P$ as in proposition \ref{prop:cyclic local phi}.  We may assume that these frames satisfy $e_0^P\wedge\dots\wedge e_{k-1}^P=e$, since if not we can multiply them by a non-vanishing local holomorphic function so that they do.  We then choose hermitian metrics $h_0$, $h_\infty$ near each point as in equation \eqref{def:h0}.  The induced metrics on $\Lambda^k(E)$ are
\begin{equation}
\det(h_0)=|w|^{2\sum_i\alpha_0^i},\quad \det(h_\infty)=|w|^{-2\sum_i\alpha_\infty^i}.
\end{equation}
By equation \eqref{eq:sum weights}, $\det(h_\infty)=|w|^{2\sum_i\alpha_0^i}$.  Therefore there exists a smooth hermitian metric $h_I$ on $E$ which agrees with $h_0$ near $w=0$ and $h_\infty$ near $w=\infty$, such that the induced metric on $\Lambda^k(E)$ is
\begin{equation}
\det(h_I)=|w|^{2\sum_i\alpha_0^i}
\end{equation}

By construction, $\mathcal{F}^{h_I}$ vanishes in neighbourhoods of the parabolic points so has compact support.  Therefore $|\mathcal{F}^{h_I}|_{h_I}$ is bounded, so $h_I$ satisfies the hypotheses of Simpson's theorem.  Simpson's theorem gives a hermitian metric $h=h_I(\cdot,e^{\psi}\cdot)$ such that $\det(h)=\det(h_I)=|w|^{2\sum_i\alpha_0^i}$, $h$ and $h_I$ are mutually bounded, $|D''\psi|_{h_I}$ is square integrable, and
\begin{equation}
\mathcal{F}^{h}-\frac{1}{k}\Tr(\mathcal{F}^{h}) = 0.
\end{equation}
Since $\Tr\mathcal{F^h}=\Tr(F^h)=\bar{\pa}\pa\ln|w|^{2\sum_i\alpha_0^i}=0$, $\mathcal{F}^h=0$ and the metric $h$ is hermitian-Einstein.  Since $h_I$ is compatible with the parabolic structure and $h,h_I$ are mutually bounded, $h$ is also compatible with the parabolic structure.  Finally, since $h,h_I$ are mutually bounded, the section $\psi$ is bounded with respect to $h_I$ and proposition \ref{prop:decay 2} gives that the curvature $F^h$ decays faster than any exponential function.

Now we establish uniqueness statement of part (i), following a standard argument.  Suppose that $h_1=h_I(\cdot,e^{\psi_1}\cdot)$ and $h_2=h_I(\cdot,e^{\psi_2}\cdot)$ are two hermitian-Einstein metrics which are compatible with the parabolic structures.  The condition of compatibility ensures that the self-adjoint sections $\psi_1,\psi_2$ are bounded with respect to $h_I$.  Then $|D''\psi_i|_{h_I}$ and $|D'' D'_{h_I}\psi_i|_{h_I}$ have bounded integrals by proposition \ref{prop:decay 1}.

Let $\psi$ be the unique self-adjoint section of $\End(E)$ such that $h_2=h_1(\cdot,e^{\psi}\cdot)$.  Let $h_t=h_1(\cdot,e^{(t-1)\psi}\cdot)$ and let $f:[1,2]\to \RR$ be the function $f(t)=DS(h_I,h_t)$.  Then $f(t)$ has critical points at $t=1,2$.  By a result of Simpson \cite{simpson},
\begin{equation}
DS(h_I,h_{t+s}) = DS(h_I,h_t)+DS(h_t,h_{t+s}).
\end{equation}
Therefore
\begin{equation}
\frac{\dd^2f}{\dd t^2} = \int_{\CP^1\setminus\{0,\infty\}} \Tr(\bar{\pa}\psi\wedge\ast (\bar{\pa}\psi)^{\ast h_t} + [\Phi,\psi]\wedge\ast[\Phi,\psi]^{\ast h_t} )\geq 0.
\end{equation}
Since $f$ has critical points at $t=1,2$, $\dd^2 f/\dd t^2=0$.  It follows that $\psi$ is holomorphic and commutes with $\Phi$.

We claim that these conditions imply that $\psi$ is a multiple of the identity.  We could prove this using stability, but in the present case it is simpler to make a direct argument.  Consider the characteristic polynomial $\det(\zeta-\psi(w))$ of $\psi$.  The coefficients are holomorphic real-valued functions of $w\in\CC^\ast$, and hence constant.  Therefore the eigenvalues $\lambda_0,\ldots,\lambda_{k-1}$ of $\psi(w)$ are independent of $w$.  Recall that near a parabolic point the eigenvalues of $\phi$ are distinct.  Since $\psi$ commutes with $\phi$ the eigenspaces of $\phi$ are subspaces of eigenspaces of $\psi$.  Recall that the eigenspaces of $\phi$ are cyclically permuted as one circles the parabolic point.  It follows that the eigenvalues $\lambda_0,\ldots,\lambda_{k-1}$ of $\psi$ are invariant under some cyclic permutation, hence equal.  Therefore $\psi$ is proportional to the identity operator and $h_2$ is a rescaling of $h_1$.

\end{document}